\documentclass[floatfix,nofootinbib,longbibliography,notitlepage]{revtex4-1}

\usepackage{amsmath,amsthm,amsfonts,amssymb,times,bbm,graphicx,mathtools,xcolor,url}

\newcommand*{\mcal}{\mathcal}

\newcommand*{\mrm}{\mathrm}

\def\CC{\mathbbm{C}}
\def\FF{\mathbbm{F}}

\def\symp{\mathrm{Sp}}
\def\cliff{\mathcal{C}l}

\def\Id{\mathbbm{1}}
\def\id{\mathrm{id}}

\def\ii{\mathbbm{1}}

\def\rep{\mathcal{K}}
\def\eps{\varepsilon}
\def\End{\mathrm{End}}

\DeclareMathOperator{\tr}{tr}
\DeclareMathOperator{\Tr}{Tr}
\DeclareMathOperator{\rank}{rank}

\DeclareMathOperator{\supp}{supp}

\DeclareMathOperator{\Sp}{Sp}
\DeclareMathOperator{\Gl}{GL}
\DeclareMathOperator{\GL}{GL}
\DeclareMathOperator{\Hom}{Hom}
\DeclareMathOperator{\range}{range}
\DeclareMathOperator{\rad}{rad}

\newtheorem{theorem}{Theorem}[section]
\newtheorem{lemma}[theorem]{Lemma}
\newtheorem*{lemma*}{Lemma}

\newtheorem{proposition}[theorem]{Proposition}
\newtheorem*{proposition*}{Proposition}
\newtheorem{corollary}[theorem]{Corollary}

\theoremstyle{definition}
\newtheorem{definition}[theorem]{Definition}

\theoremstyle{remark}

\numberwithin{equation}{section}

\begin{document}

\title{Rank-deficient representations in the
Theta correspondence over finite fields\\
arise from quantum codes}

\author{Felipe Montealegre-Mora}
\email{fmonteal@thp.uni-koeln.de}

\author{David Gross}
\affiliation{\
Institute for Theoretical Physics, University of Cologne, 50937 Cologne, Germany}

\thanks{We thank Mateus Araujo, Markus Heinrich, Sepehr Nezami, and Michael Walter for interesting discussions.
This work has been supported by the Excellence Initiative of the German Federal and State Governments (Grant ZUK 81), the ARO under contract W911NF-14-1-0098 (Quantum Characterization, Verification, and Validation),
and the DFG (SPP1798 CoSIP, project B01 of CRC 183).}

\date{\today}

\begin{abstract}
  Let $V$ be a symplectic vector space and let $\mu$ be the \emph{oscillator representation} of $\Sp(V)$.
  It is natural to ask how the tensor power representation $\mu^{\otimes t}$ decomposes. 
  If $V$ is a real vector space, then \emph{Howe-Kashiwara-Vergne (HKV) duality} asserts that there is a one-one correspondence between the irreducible subrepresentations of $\Sp(V)$ and the irreps of an orthogonal group $O(t)$.
  It is well-known that this duality fails over finite fields.
  Addressing this situation, Gurevich and Howe have recently assigned a notion of \emph{rank} to each $\Sp(V)$ representation.
  They show that a variant of HKV duality continues to hold over finite fields, if one restricts attention to subrepresentations of maximal rank.
  The nature of the rank-deficient components was left open.
  Here, we show that all rank-deficient $\Sp(V)$-subrepresentations arise from embeddings of
  lower-order tensor products of $\mu$ and $\bar\mu$ into $\mu^{\otimes t}$.
  The embeddings live on spaces that have been studied in quantum information theory as tensor powers of \emph{self-orthogonal Calderbank-Shor-Steane (CSS) quantum codes}.
  We then find that the  irreducible $\Sp(V)$ subrepresentations of $\mu^{\otimes t}$ are labelled by
  the irreps of orthogonal groups $O(r)$ acting on certain $r$-dimensional spaces for $r\leq t$.
  The results hold in odd charachteristic and the ``stable range'' $t\leq \frac12 \dim V$.
  Our work has implications for the representation theory of the \emph{Clifford group}.
  It can be thought of as a generalization of the known characterization of the invariants of the Clifford group in terms of self-dual codes.
\end{abstract}

\maketitle


\section{Introduction and summary of results}
\label{sec:summary}

The \emph{oscillator representation}  (also: \emph{Schr\"odinger}, \emph{Weil}, or \emph{metaplectic} representation) is a representation $\mu_V$ of the symplectic group $\Sp(V)$ over a symplectic vector space $V$.
It appears in many contexts, including time-frequency analysis,
coding theory,
and quantum mechanics. 

The starting point of this work is the natural question of how tensor powers $\mu_V^{\otimes t}$
decompose into irreducible representations. 

One may reformulate this problem in a more geometric and slightly more general way~\cite{gh17}.
If $U$ is an orthogonal space, then $U\otimes V$ is again symplectic.
The tensor power $\mu_V^{\otimes t}$ is isomorphic to $\mu_{U\otimes V}$ for a suitable $t$-dimensional space $U$ (Corollary~\ref{cor:full factorization}).
The symmetry group $O(U)\times \Sp(V)$ associated with the tensor factors embeds into $\Sp(U\otimes V)$.
Clearly, the restriction of $\mu_{U\otimes V}$ to $O(U)$ commutes with the restriction to $\Sp(V)$.
One can thus decompose the representation into a direct sum
\begin{align}\label{eqn:Theta}
  \mu_{U\otimes V} \cong \bigoplus_{\tau \in\mrm{Irr}(O(U))} \tau\otimes\Theta(\tau),
\end{align}
where $\tau$ ranges over irreps of $O(U)$, and $\Theta(\tau)$ is a representation of $\Sp(V)$.
If $V$ is a real vector space, then \emph{HKV duality} asserts that $\Theta(\tau)$ is again 
irreducible, and that the correspondence $\Theta$ between representations is 
injective~\cite{howe89,kashiwaravergne78}.
Over finite fields, the correspondence fails: $\Theta(\tau)$ is in general no longer irreducible, 
and equivalent $\Sp(V)$ representations might appear in $\Theta(\tau)$ for different $\tau$'s.
Our goal is to understand this situation better.

The main part of this paper is presented in the basis-free notation set out in Ref.~\cite{gh17}.
For ease of exposition, we will use more concrete (and slightly less general) constructions in this introductory section.
From now on, we assume that $V=\FF_q^{2n}$ is $2n$-dimensional over a finite field $\FF_q$ of odd characteristic, and endowed with a symplectic form.

Reference~\cite{gh17} introduces a notion of \emph{rank} for $\Sp(V)$ representations.
To describe it, recall  that the oscillator representation of $\Sp(V)$ can be realized over the Hilbert space $\mathcal{H}=\CC[\FF_q^{n}]$ of complex linear combinations of basis vectors $\delta_x$ labeled by vectors $x\in\FF_{q}^n$ (Sec.~\ref{sec:background}).
Given $x_1, \dots, x_t \in \FF_q^n$, we may arrange these vectors as the rows of a $t\times n$ matrix $F$.
We then obtain an isormorphism
\begin{align}\label{eqn:to matrices}
  \mathcal{H}^{\otimes t} = \CC[\FF_q^n]^{\otimes t} \simeq \CC[\FF_q^{t \times n}]
\end{align}
via the identification
\begin{align*}
  \delta_{x_1}\otimes \dots \otimes \delta_{x_t} \simeq \delta_F.
\end{align*}
The \emph{rank} of an element $\psi\in\mathcal{H}^{\otimes t}$ and of a subspace $\mathcal{K}\subset\mathcal{H}^{\otimes t}$ are defined as, respectively,
\begin{align*}
  \rank\psi = \sup\,\{ \rank F^T F \;|\; (\delta_F, \psi) \neq 0 \},
  \qquad
  \rank\mathcal{K}= \sup\,\{ \rank \psi \;|\; \psi \in \mathcal{K} \}.
\end{align*}
The central result of Ref.~\cite{gh17} is this:

\begin{theorem}[\cite{gh17}]\label{thm:gurevich}
  Assume $t\leq n$.  
  Then $\Theta(\tau)$ contains a unique irreducible representation $\eta(\tau)$ of rank $t$.
  The function $\eta$ defines an injective map from the irreducible representations of $O(U)$ to the irreducible rank-$t$ subrepresentations of $\Sp(V)$ in $\mu_{U\otimes V}$.
\end{theorem}

The purpose of this work is to understand the \emph{rank-deficient} $\Sp(V)$-subrepresentations of $\mu_{U\otimes V}$, i.e.\ those of  that have rank $r<t$.
Key to this are \emph{self-orthogonal Calderbank-Shor-Steane (CSS) quantum codes}
\cite{steane1996error,calderbank1996good,steane1996multiple},
which are studied in the theory of quantum error correction \cite{nielsen-chuang}.
For now, we will take $U=\FF_q^t$ with the standard orthogonal form $\beta(u,v)=\sum_{i=1}^t u_i v_i$.
Let $N$ be an \emph{isotropic subspace} of $U$, i.e.\ such that $N\subset N^\perp$.
To each coset $[u]=u+N\subset U$ of $N$, one associates the \emph{coset state}
\begin{align*}
  e_{[u]} = \sum_{v\in [u]} \delta_{v} \in \CC[\FF_q^t].
\end{align*}
Analogous to the construction in Eq.~\eqref{eqn:to matrices}, we identify
\begin{align*}
  \CC[\FF_q^t]^{\otimes n} \simeq \CC[\FF_q^{t \times n}],
  \qquad
  \delta_{u_1}\otimes \dots \otimes \delta_{u_n} \simeq \delta_F,
\end{align*}
where $F$ is now the matrix whose \emph{columns} are given by the $u_1, \dots, u_t \in \FF_q^n$.
The \emph{tensor power CSS code} $C_N$ associated with $N$ is the space with basis
\begin{align}\label{eqn:coset powers}
  \left\{ e_{[u_1]} \otimes \dots \otimes e_{[u_n]}  \;\big|\; [u_i] \in N^\perp / N \right\},
\end{align}
the set of products of coset states corresponding to the elements of the quotient space $N^\perp /N$.

\begin{figure}
	\includegraphics{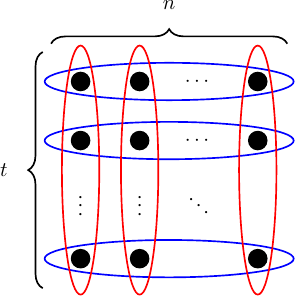} 
\caption{
\label{fig:commuting action}
Sketch of the commuting actions of the Weil representation and tensor-power CSS codes. 
The each tensor factor in the representation $\mu_V^{\otimes t}(S)$ for an arbitrary 
$S$ acts on a row (highlighted in blue). The code projector is an $n$-th tensor power 
of a projector supported on a column (red).
}
\end{figure}

The codes $C_N$ can be shown to be invariant subspaces of $\mu_{U\otimes V}|_{\Sp(V)}$.
What is more, we will show:

\begin{lemma*}[Lemma~\ref{lem:css representation}, simplified version]
  As a representation of $\Sp(V)$, the restriction of $\mu_{U\otimes V}$ to a tensor power CSS code $C_N$ is isomorphic to $\mu_{U' \otimes V}$, where $U'=N^\perp / N$.
\end{lemma*}

Figure~\ref{fig:commuting action} displays graphically the commuting actions of the projector
$P_N$ onto an arbitrary CSS code $C_N$ and $\mu_V^{\otimes t}(S)$ for an arbitrary $S$. 
There, we identify
\begin{align}
  \label{eqn:full factorization intro}
  \CC[\FF_q]^{\otimes nt}\simeq \CC[\FF_q^{t\times n}],
\end{align}
in an analogous way as above. Each dot in the diagram corresponds to a $\CC[\FF_q]$ factor
in the left-hand side of~\eqref{eqn:full factorization intro}. The tensor factors in the 
Weil representation 
$\mu_V^{\otimes t}$ act row-wise, highlighted in blue, whereas the projector $P_N$ acts 
column-wise, highlighted in red.

(We note that in odd characteristic, there are two inequivalent orthogonal geometries in each dimension.
They are distinguished by their \emph{discriminant}, the square class of the determinant of the Gram matrix of any basis.
So far, we have only considered the standard orthogonal form on $U=\FF_q^t$.
It turns out that $U'=N^\perp/N$ inherits an orthogonal form from $U$ -- however, it need not be equivalent to the standard one.
We will deal with this more general situation in the main part.)

The lemma immediately implies that non-trivial CSS codes carry rank-deficient representations of the symplectic group.
Our main result is that this construction is exhaustive.

\begin{theorem}[Main Theorem]\label{thm:main}
    Assume that $t\leq n$ and let $\rep$ be an $\Sp(V)$-subrepresentation of $\mu_{U\otimes V}$ of rank $r$.
	Then
	$(t-r)$ is even and $\mathcal{K}$ is contained in the span of all tensor power CSS codes $C_N$ with $\dim N=(t-r)/2$.
\end{theorem}

The result allows us to give an explicit decomposition of $\mu_{U\otimes V}$ in terms of irreducible and inequivalent $\Sp(V)$ representation spaces.
Indeed, we find (Sec.~\ref{sec:eta}) that as an $O(U)\times \Sp(V)$ representation:
\begin{align}\label{eqn:nice intro}
  \mu_{U\otimes V} \simeq
  \bigoplus_{r \in R(U)}\quad \bigoplus_{\tau \in\mrm{Irr}\,O(U_r)}\quad
  \mathrm{Ind}_{O_r}^{O(U)}(\tau)
  \otimes
  \eta(\tau).
\end{align}
We have used the following expressions:
$R(U)$ is the set $\{ t -2k \}_k$, where $k$ ranges from $0$ to the dimension of the largest 
isotropic subspace in $U$ (its \emph{isotropy index}).
For each $k$, we choose some istropic subspace $N\subset U$ of dimension $k$ and set $U_r=N^\perp/N$.
Then $U_r$ is an orthogonal space of dimension $r=t-2k$ and discriminant $d(U_r)=(-1)^k d(U)$.
Let $O_r:=O_N\subset O(U)$ be the stabilizer of $N$.
Notice that because of a lemma proven by Witt, the group 
$O_r$ is independent of the choice of $N$, up to isomorphism. 
This justifies surpressing $N$ in our notation.
The group $O_r$ acts on $U_r$ as $O(U_r)$.

Thus any $\tau\in\mathrm{Irr}\,O(U_r)$ can be interpreted as an $O_r$-representation, and the induced representation in Eq.~\eqref{eqn:nice intro} is hence well-defined.
All $\Sp(V)$-irreps $\eta(\tau)$ appearing in Eq.~\eqref{eqn:nice intro} are indeed inequivalent: Those corresponding to different $O(U_r)$ are distinguished by their rank, whereas the inequivalence of summands of the same rank is a consequence of Theorem~\ref{thm:gurevich}.

It is natural to ask whether the assumption that $t\leq n$ is necessary.
We show that some constraints on $t, n$ are indeed required, by explicitly constructing rank-$0$ subrepresentations for $t=3, n=1$ that do not come from CSS codes (Section~\ref{sec:counterexample}).

Our work was motivated by recent related observations on tensor powers of the \emph{Clifford group} \cite{zhu2016clifford,helsen2016representations,nezami2016multipartite,zhu2015multiqubit,kueng2015qubit,webb2016clifford,nebe_invariants,nebe-book,runge93},
the group generated by the oscillator representation of $\Sp(V)$ and the Weyl representation of the Heisenberg group.
In Ref.~\cite{grossnezamiwalter}, it has been shown that the commutant algebra of the Clifford group is generated by projections onto tensor power CSS codes whose isotropic spaces  are orthogonal to the all-ones vector $\underline{1}=(1,\dots, 1)\in\FF_q^t$; together with the elements of $O(U)$ that preserve $\underline{1}$.
While it was not explicitly worked out in Ref.~\cite{grossnezamiwalter}, their arguments strongly suggest that the commutant of the oscillator representation alone is generated by $O(U)$ and tensor power CSS codes, without the constraints involving the $\underline{1}$-vector.
This drew our attention to the action of tensor power representations on CSS code spaces.
While the present paper mostly focuses on the symplectic group alone---instead of the full Clifford group---one can in some cases relate the theory for the two groups explicitly (Sec.~\ref{sec:clifford}):

\begin{proposition*}[Proposition~\ref{prop:symplectoclifford}, simplified version]
  If the characteristic does not divide $t$, there is a one-one correspondence between irreducible components of $t$-th tensor powers of the Clifford goup and irreducible components of $\mu_{U\otimes V}$ for a certain orthogonal space $U$ of dimension $t-1$.
\end{proposition*}

An $\Sp(V)$-representation space is trivial if and only if it has rank equal to $0$ \cite{gh17}.
The rank-$0$ case connects our results with prior work on the \emph{invariants} of the Clifford group \cite{nebe-book,runge93}.
Indeed, it is well-knwon that the invariants are associated with self-\emph{dual} CSS codes, i.e.\ those arising from subspaces $N\subset U$ with $N^\perp = N$.
In this sense, our work can be seen as a generalization of these results to higher ranks.

Our main theorem is based on a careful analysis of the action of certain Fourier transforms in the oscillator representation.
The same techniques can be used to find auxilliary results, which may be of independent interest.
For example, we show that the ``set of ranks'' one can associate with an irreducible $\Sp(V)$-subrepresentation of $\mu_{U\otimes V}$ is a contiguous set of integers (Prop.~\ref{prop:rank spectrum}).

The rest of the paper is organized as follows:
We will introduce the technical background in Sect.~\ref{sec:background},
prove the main theorem in Sec.~\ref{sec:main},
and lay out the connections to the Clifford group in Sec.~\ref{sec:clifford}.

This work is written in a basis-free language inspired by \cite{gh17}.
We believe that the results will be of interest to researchers in quantum information theory, who may not be familiar with this point of view.
A follow-up paper \cite{in-prep} will address a quantum information audience, both in terms of presentation and in terms of applications.
In particular, it will also treat the Clifford group in characteristic 2.

\section{Technical background}
\label{sec:background}

In this section, we collect definitions and some technical results.

\subsection{General notation}

In what follows, $q$ is the power of an odd prime $p$, and $\FF_q$ the finite field of order $q$.
We denote the multiplicative group in $\FF_q$ by $\FF_q^\times$.
For $\lambda\in\FF_q^\times$, the \emph{Legendre symbol} is
 $\big(
   \frac{\lambda}{q}
 \big)$,
which is $+1$ if $\lambda$ is a square in $\FF_q^\times$, and $-1$ else.
If $q$ is clear from the context, we also use the short-hand notation $\ell_\lambda$ for the Legendre symbol.
We write $\Tr$ for the the \emph{field trace} $\FF_q \to \FF_p$.

The \emph{transpose} of a linear map $A: Y\to Z$ is $A^t: Z^* \to Y^*$ (not to be confused with $A^T$, which is defined in Eq.~\eqref{eqn:Transpose}).
A map $A: Y \to Y^*$ is \emph{symmetric} if $A=A^t$.

\subsection{The oscillator representation}
\label{sec:oscillator}

Let $V=X\oplus X^*$ be the direct sum of two $n$-dimensional dual vector spaces over $\FF_q$.
The space $V$ carries a symplectic form
\begin{equation*}
  [x\oplus y, x'\oplus y']= y'(x)-y(x').
\end{equation*}
Every symplectic vector space is (non-canonically) of this form.
Indeed, the choice of a decomposition $V=X\oplus X^*$ is equivalent to fixing a \emph{polarization} of $V$.
From now on, we will assume that dual $X, X^*\subset V$ have been chosen.

The \emph{oscillator representation} $\mu_{V}$ is a representation of $\Sp(V)$ on the Hilbert space $L^2(X^*)$ of complex functions on $X^*$.
The representation depends on a parameter $m\in\FF_q^\times$
-- sometimes referred to as the \emph{mass} of the representation in mathematical physics \cite{folland_harmonic} --
which defines a character
\begin{align*}
  \omega^{(m)}: \FF_q \to \CC, \qquad \lambda \mapsto e^{i\frac{2\pi }{p} \Tr (m \lambda )}
\end{align*}
of $\FF_q$.
One can show \cite{gh17} that the oscillator  representations $\mu_V^{(m)}$, $\mu_V^{(m')}$ are unitarily equivalent if and only if $m$ and $m'$ belong to the same square class.
What is more, $\mu_V^{(-m)} = \bar{\mu}_V^{(m)}$, i.e.\ the inverting the sign of the mass corresponds to passing to the complex conjugate representation.
From now on, we will write $\omega, \mu_V$ for $\omega^{(1)}$ and $\mu_V^{(1)}$ respectively.

Next, we recall \cite{gh17,gerardin} the explicit form of the oscillator representation on the following three subsets, which taken together generate $\Sp(V)$.
\begin{align}
  \mathcal{J} &=
  \left\{
  \left(
       \begin{array}{cc}
         0 & B \\
         -B^{-1} & 0
       \end{array}
  \right)
	\,\Big|\,
	B: X^*\to X,\; B\text{ invertible, symmetric}
  \right\}, \label{eqn:J} \\
  \mathcal{N} &=
  \left\{
	\left(
	  \begin{array}{cc}
		\Id & A \\
		0 & \Id
	  \end{array}
	\right)
	\,\Big|\,
	A: X^*\to X,\; A\text{ symmetric}
  \right\}, \label{eqn:N} \\
  \mathcal{D} &=
  \left\{
	\left(
	  \begin{array}{cc}
		C & 0 \\
		0 & C^{-t}
	  \end{array}
	\right)
	\,\Big|\,
	C\in GL(X)
  \right\}. \label{eqn:D}
\end{align}
The sets $\mathcal{N}$ and $\mathcal{D}$ are subgroups and generate the \emph{Siegel parabolic}, with the Abelian $\mathcal{N}$ the \emph{unipotent radical} of the parabolic group.
We write, respectively, $N_A, J_B, D_C$ for the elements of $\mathcal{N}, \mathcal{J}, \mathcal{D}$ that appear above.
Let $y\in X^*$ and let $\delta_y\in L^2(X^*)$ the indicator function at $y$.
Then the action of the oscillator representation is 
\begin{align}
    \mu_V\left(J_B\right) \delta_y
    &=
    \gamma(B)^{-1} \sum_{y'\in X^*}
    \omega\big(-2^{-1} B(y,y')\big) \delta_{y'}, \label{eqn:J oscillator}\\
    \mu_V\left(N_A\right)\delta_y
    &=
    \omega\big(2^{-1} A(y,y)\big)\,\delta_y, \label{eqn:N oscillator}\\
  \mu_V\left(D_C\right)\delta_y
	&=
	\ell_{\det C}\,\delta_{C^{-t} y}, \label{eqn:D oscillator}
\end{align}
where $B(y,y')$ is a less-confusing notation for $B(y)(y')$, and where
\begin{align*}
  \gamma(B) = \sum_{y\in X^*} \omega\big(-2^{-1}B(y,y)\big)
\end{align*}
is the \emph{Gauss sum} corresponding to $B$.

We will frequently make use of the fact that the oscillator representation of block matrices factorizes.

\begin{lemma}\label{lem:factorization general}
  Let $X=X_1\oplus X_2$ be a direct sum of vector spaces.
  Then we have an orthogonal decomposition $V=V_1\oplus V_2$ of $V=X\oplus X^*$ into symplectic subspaces $V_i=X_i\oplus X_i^*$.
  As a representation of the subgroup $\Sp(V_1)\times\Sp(V_2) \subset \Sp(V)$, the oscillator representation factorizes
  \begin{align}\label{eqn:factorization general}
	{\mu_V}\simeq \mu_{V_1}\otimes \mu_{V_2}.
  \end{align}

  Let $\pi_i: X \to X_i$ be the projections onto the $i$-th direct summand.
  An isomorphism
  \begin{align*}
	L^2(X^*) \to L^2(X^*_1) \otimes L^2(X^*_2)
  \end{align*}
  realizing Eq.~\eqref{eqn:factorization general} is given by
  \begin{align}
	\delta_y \mapsto \delta_{y \pi_1}\otimes\delta_{y\pi_2}. \label{eqn:isomorphism factorization}
  \end{align}
\end{lemma}

This factorization property is well-known -- see e.g.\ Ref.~{\cite[Corollary 2.5]{gerardin}}.
We give a short self-contained proof in Appendix~\ref{app:factorization proof}.

\subsection{The rank of a representation}
\label{sec:rank}

We consider the subgroups $\mathcal{N}, \mathcal{D}$ of $\Sp(V)$ given in Eqs.~\eqref{eqn:N}, \eqref{eqn:D}.

If $\pi$ is a representation of $\Sp(V)$ on some Hilbert space $\mathcal{H}$, then the restriction of $\pi$ to the Abelian group $\mathcal{N}$ decomposes $\mathcal{H}$ into a direct sum of one-dimensional representations.
Every character of $\mathcal{N}$ is of the form
\begin{align*}
  N_A \mapsto \omega\big(\tr A B\big)
\end{align*}
for some symmetric $B: X \to X^*$, which we will refer to as an $\mcal{N}$-\emph{weight}.
With each representation space $\CC \Phi\subset \mathcal{H}$, we can thus associate an $\mcal{N}$-weight $B$ such that
\begin{align*}
  \pi(N_A)\Phi = \omega( \tr AB )\Phi, \quad \forall\, N_A\in\mcal{N}.
\end{align*}
Reference~\cite{gh17} defines the \emph{$\mathcal{N}$-spectrum} of $\pi$ as the set of $\mathcal{N}$-weights,
counted with multiplicities, that occur in the decomposition of $\mathcal{H}$.

The set of $\mathcal{N}$-weights decomposes into a union of orbits under the action $B \mapsto C B C^t$, $C \in \GL(X)$.
This follows from the fact that $\mathcal{D}$ normalizes $\mathcal{N}$:
\begin{align*}
  D_C N_A D_C^{-1} = N_{C A C^{-t}},
\end{align*}
so that if $\Phi$ carries the $\mathcal{N}$-weight $B$, then $\pi(S_C)\Phi$ is associated with the $\mathcal{N}$-weight $C^{-t} B C$.
From the theory of quadratic forms, it is well-known that the orbits are labelled by the rank and the discriminant of $B$
(c.f.\ Section~\ref{sec:orthogonal}).

The \emph{rank of $\pi$} is the maximum of the rank taken over the $\mathcal{N}$-spectrum.
If all $\mathcal{N}$-weights of maximal rank have the same discriminant $d$, $\pi$ is said to have \emph{discriminant} or  \emph{type $d$}.

As an example, we compute the $\mathcal{N}$-spectrum of the oscillator representation.
By Eq.~\eqref{eqn:N oscillator}, the delta functions $\{\delta_y \;|\; y\in X^*\}$ diagonalize the restriction of $\mu_V$ to $\mathcal{N}$.
We can re-write
\begin{align*}
  A(y,y) = A(y)(y) = \tr A (y\otimes y).
\end{align*}
The map $B=2^{-1}\,y\otimes y$ is the most general form of a symmetric map $X\to X^*$ of rank $\leq 1$ and of discriminant $\ell_2$.
Since $\pm y$ lead to the same $B$, the $\mathcal{N}$-spectrum consists of the following $\GL(X)$-orbits: $\{0\}$ occurs once, and the set of non-zero rank-1 $B$'s of discriminant $\ell_2$ occurs twice.

\subsection{Orthogonal spaces and higher-rank representations}
\label{sec:product spaces}

We recall some standard facts about discrete orthogonal spaces (see e.g. \cite{lam2005introduction,pjcameron,chan_notes}) and fix notation.

Let $U$ with be a $t$-dimensional $\FF_q$-vector space with non-degenerate symmetric form $\beta$.
Let $\{f_i\}_{i=1}^t$ be a basis of $U$.
The square class $d(U)$ of the determinant of the matrix with elements $\beta(f_i, f_j)$ does not depend on the basis.
It is called the \emph{discriminant} of the form $\beta$.
Quadratic spaces are characterized up to isometries by their dimension and discriminant.
The discrimant is multiplicative:  if $U_1\oplus U_2$ is an orthogonal sum, then
\begin{align*}
  d(U_1 \oplus U_2) = d(U_1)d(U_2).
\end{align*}
One can find an \emph{orthogonal basis} that diagonalizes the form in that
\begin{align}
	\label{eqn:orthogonal basis}
  	\beta(f_i,f_j) = d_i\,\delta_{i,j}
\end{align}
for suitable $d_i\in\FF_q$.
From the discussion above, it follows that one can choose
\begin{align}\label{eqn:discriminant}
  d_i &= 1\quad (i=1, \dots, t-1),\qquad d_t \in d(U),
\end{align}
and we will usually do so.

An important orthogonal space is the \emph{hyperbolic plane} $\mathbb{H}$, which has dimension $t=2$ and discriminant $d(\mathbb{H})=-1$.

For a subspace $N\subset U$, the space orthogonal to it is $N^\perp = \{ u \;|\; \beta(u,v)=0 \,\forall v \in N\}$.
The space $N$ is \emph{isotropic} if $N \subset N^\perp$. 
From the relation $\dim N + \dim N^\perp = t$, valid for any non-degenerate form, one finds the dimension
bound for isotropic spaces:
\begin{align}\label{eqn:isotropic dimension bound}
  N\subset N^\perp \qquad\Rightarrow\qquad
  \dim N \leq \frac t2
\end{align}

We will use the symbol $\beta$ both to refer to the form $U\times U \to \FF_q$ and to the induced isomorphism
\begin{align*}
  \beta: U \to U^*, \qquad u \mapsto \beta(u):=\beta(u, \cdot).
\end{align*}
For maps $F\in\Hom(Y\to U)$, we will write
\begin{align}\label{eqn:Transpose}
  F^T :=  F^t \circ \beta \in\Hom(U \to Y^*).
\end{align}
With $U \simeq \Hom(\FF_q\to U)$ and $\FF_q^*\simeq \FF_q$, this implies in particular
\begin{align*}
  u^T=u\circ \beta 
  = \beta(u)
  =\beta(u,\cdot).
\end{align*}

If the form $\beta$ is degenerate, then the quotient space $U/\rad\beta$ of $U$ up to the radical of $\beta$ is non-degenerate.
The \emph{rank} and the discriminant of $U$ are then defined to be the dimension and the discriminant of the quotient space.

A symmetric map $B: X\to X^*$ defines a quadratic form $B(x,y)=B(x)(y)$ on a linear space $X$.
Below, we will often be concerned with forms defined as $B=F^T F$ for some $F: X \to U$.
In this case, we have
\begin{align}\label{eqn:b=yty}
  B(x,y)
  =(F^t\beta F)(x)(y)
  =\beta(Fx, Fy),
\end{align}
so that the rank and discriminant of such $B$ are the rank and the discriminant of $\range F$ as a subspace of $U$.

Given a space $V=X \oplus X^*$ and an orthogonal space $U$, the tensor product $U\otimes V$ is again a direct sum of dual spaces and thus carries a symplectic form.
Indeed,
\begin{align}\label{eqn:product configuration space}
  U\otimes V
  \simeq (U\otimes X) \oplus (U\otimes X^*)
\end{align}
and the pairing between (factorizing) elements of the two summands is just
\begin{align}\label{eqn:pairing}
  \langle u\otimes x, v\otimes y \rangle  =  \beta(u,v) y(x).
\end{align}
We will usually make the identification
\begin{align*}
  U\otimes X &= \Hom(X^*\to U),
  \qquad
  U\otimes X^* = \Hom(X\to U).
\end{align*}
Then the pairing~\eqref{eqn:pairing} between $Z\in\Hom(X^*\to U)$ and $F\in\Hom(X\to U)$ takes the form
\begin{align}\label{eqn:operator pairing}
  \langle Z, F \rangle
  = \tr \beta Z F^t.
\end{align}
It follows that there is an oscillator  representation $\mu_{U\otimes V}$ of $\Sp(U\otimes V)$ on $L^2(\Hom(X \to U))$.

From Eq.~\eqref{eqn:pairing}, one sees that $O(U)\times \Sp(V)$ embeds into $\Sp(U\otimes V)$.
The main goal of this work is to understand the restriction of $\mu_{U\otimes V}$ to $\Sp(V)$.

We compute the $\mathcal{N}$-spectrum and rank of $\mu_{U\otimes V}$ as an $\Sp(V)$-representation.
To this end, we must find the eigenspaces of $\mu_{U\otimes V}(\Id_U\otimes N_A)$.
Under the identification~\eqref{eqn:product configuration space},
\begin{align*}
  N_A
  =
  \left(
	\begin{array}{cc}
	  \Id & A \\
	  0 & \Id
	\end{array}
  \right)
  \in \Sp(V)
  \qquad
  \Rightarrow
  \qquad
  \Id \otimes N_A
  \simeq
  \left(
	\begin{array}{cc}
	  \Id\otimes\Id & \Id \otimes A \\
	  0 & \Id\otimes\Id
	\end{array}
  \right) \in \Sp(U\otimes V).
\end{align*}
Thus, the embedding $\Id\otimes N_A$ of $N_A\in\Sp(V)$ into $\Sp(U\otimes V)$ is again an element of the unipotent radical.
The action of $\mu_{U\otimes V}(\Id\otimes N_A)$ is thus also given by Eq.~\eqref{eqn:N oscillator}, this time acting on $L^2(\Hom(X\to U))$.
Let $F\in\Hom(X\to U))$.
With Eq.~\eqref{eqn:operator pairing}, we can express the quadratic form in Eq.~\eqref{eqn:N oscillator} as
\begin{align}\label{eqn:embedded weight}
  (\Id\otimes A)(F)(F)
  =\langle F, (\Id\otimes A) F \rangle
	= \langle F, F A \rangle
	= \tr \beta F A F^t
    = \tr F^T F A.
\end{align}
The $\mathcal{N}$-weight on $\delta_F$ is thus given by
\begin{align*}
  B=2^{-1} F^T F.
\end{align*}
Its rank is upper-bounded by $\min(n,t)$.
From now on, we will focus on the case where $t\leq n$ (this is referred to as the \emph{stable range} in \cite{gh17}), and call a representation of rank strictly smaller than $t$ \emph{rank-deficient}.
It follows that the representation space
\begin{align}\label{eqn:Y weights}
  \big\{ \Phi \in L^2(\Hom(X\to U)) \;\big|\; \mu_{U\otimes V}(N_A)\Phi = \omega(\tr AB)\Phi \big\}
\end{align}
on which $\mathcal{N}\subset \Sp(V)$ acts with $\mathcal{N}$-weight $B$ is equal to the span
$
  \langle
    \{ \delta_F \;|\; F^T F = B\}
  \rangle
$
of the $\delta_F$'s with $F^T F=B$.

\subsection{Representations associated with direct sums of orthogonal spaces}
\label{sec:orthogonal}

The original motivation of this work was to understand tensor power representations $\mu_V^{\otimes t}$.
The more geometric language employed e.g.\ in Ref.~\cite{gh17} relates tensor factors to direct summands of orthogonal spaces.
The following corollary of Lemma~\ref{lem:factorization general} makes the connection precise.

\begin{corollary}
  \label{cor:factorization U}
 Assume $U=U_1 \oplus U_2$ is an orthogonal direct sum.
 Then, as a representation of $\Sp(V)$, the oscillator representation factorizes as
 \begin{align}\label{eqn:orthogonal factorization}
  \mu_{(U_1\oplus U_2) \otimes V } \simeq \mu_{U_1 \otimes V }\otimes \mu_{U_2\otimes V}.
 \end{align}
 Let  $\pi_i: U \to U_i$ be the projections onto the direct summands.
 An isomorphism
 \begin{align*}
  L^2(\Hom(X\to U)) \to L^2(\Hom(X\to U_1))\otimes L^2(\Hom(X \to U_2)).
 \end{align*}
 realizing Eq.~\eqref{eqn:orthogonal factorization} is defined by
 \begin{align}\label{eqn:orthogonal isomorphism}
  \delta_F \mapsto \delta_{\pi_1 F} \otimes \delta_{\pi_2 F}.
 \end{align}
\end{corollary}

\begin{proof}
	By assumption, both terms $U_i$ are non-degenerate $\beta$-spaces, so we have a canonical identifications $U_i^*\cong U_i$
	and $\Hom(X\to U_i)^* \cong \Hom(X^*\to U_i)$. The latter identification satisfies that for any $h\in\Hom(X\to U_1)^*$ and
	any $f \in\Hom(X\to U_2)$, it holds that $h(f)=0$ (and the same statement holds if we exchange $U_1$ and $U_2$).

	This way, the advertised claim is a consequence of Corollary~\ref{lem:factorization general} for the decomposition
	\begin{align*}
		\Hom(X\to U) 	&= \Hom(X\to U_1)\oplus\Hom(X\to U_2)\\
		\Hom(X\to U)^*&= \Hom(X^*\to U_1)\oplus\Hom(X^*\to U_2)\\
	\end{align*}
	which give rise to the following decomposition into symplectic subspaces
	\begin{align*}
		U\otimes V = (U_1\otimes V)\oplus(U_2\otimes V).
	\end{align*}
\end{proof}

Iterating this observation over an orthogonal basis gives the connection between  $\mu_{U\otimes V}$ and tensor powers of $\mu_V$.

\begin{corollary}\label{cor:full factorization}
  As a representation of $\Sp(V)$, we have that
  \begin{align}\label{eqn:full factorization}
	\mu_{U\otimes V}
	\simeq
	\underbrace{\mu_{V}\otimes\dots\otimes\mu_V}_{ (t-1) \, \times }\otimes \mu_V^{(d(U))}.
  \end{align}

  Let $\{f_i\}_{i=1}^t$ be an orthogonal basis of $U$ as in Eq.~\eqref{eqn:discriminant}.
  An isomorphism
  \begin{align*}
	L^2(\Hom(X\to U)) \to \big(L^2(X^*)\big)^{\otimes t}
  \end{align*}
  realizing Eq.~\eqref{eqn:full factorization} is defined by
  \begin{align}\label{eqn:full isomorphism}
	\delta_F \mapsto \delta_{f_1^T F} \otimes \dots
	\otimes \delta_{f_t^T F}.
  \end{align}
\end{corollary}

\begin{proof}
  Set $U_i = \FF_q f_i$, so that $d(U_i)=\beta(f_i,f_i)=d_i$.
  The projections $\pi_i: U\to U_i$ are given by
  \begin{align*}
	u\mapsto d_i^{-1}\,f_i\,f_i^T(u).
  \end{align*}
  Iterating Corollary~\ref{cor:factorization U} thus gives an isomorphism
  \begin{align*}
	i_1: L^2(\Hom(X\to U)) \to \bigotimes_{i=1}^t L^2(\Hom(X\to U_i))
  \end{align*}
  defined by
  \begin{align*}
	\delta_F \mapsto \delta_{f_1 f_1^T F} \otimes \dots \otimes\delta_{f_{t-1}f_{t-1}^T F}\otimes \delta_{d(U)^{-1}\,f_t f_t^T F}.
  \end{align*}
  We may identify $\Hom(X\to U_i)\simeq U_i\otimes X^*$ with $X^*$ via $f_i\otimes y\mapsto y$.
  This induces an isomorphism
  \begin{align*}
	i_2:
	\bigotimes_{i=1}^t L^2(\Hom(X\to U_i)) \to \big(L^2(\Hom(X^*)\big)^{\otimes t}.
  \end{align*}
  Finally, let $C = d(U)^{-1} \Id \in GL(X)$ and, using Eq.~\eqref{eqn:D oscillator}, let $i_3$ be $\mu_V(D_C)$ acting on the $t$-th tensor factor.
  Then the advertised isomorphism is $i_3\, i_2\,  i_1$.
\end{proof}

Note that the standard inner product $\beta(x,y) = \sum_{i=1}^t x_i y_i$ on $\FF_q^t$ has an ortho\emph{normal} basis, and thus discriminant $d(\FF_q^t)=1$.
Therefore,
\begin{align}\label{eqn:tensorpower equiv}
 \mu_{\FF_p^t\otimes V}
 \simeq
 \mu_V^{\otimes t}.
\end{align}

We end this section by analyzing $\mu_{\mathbb{H}\otimes V}$, where $\mathbb{H}$ is the hyperbolic plane.
To this end, define the \emph{permutation representation $\pi$} of $\Sp(V)$ as the map that acts on $L^2(V)$ by sending the delta function $\delta_v$ at $v\in V$ to
\begin{align}\label{eqn:permutation}
  \pi(S)\delta_v = \delta_{Sv}.
\end{align}

\begin{lemma}\label{lem:hyperbolic plane}
  Let $\mathbb{H}$ be the hyperbolic plane. We then have:
  \begin{enumerate}
	\item
	As a representation of $\Sp(V)$, $\mu_{\mathbb{H}\otimes V}$
	  is isomorphic to the permutation representation.

	\item
	If $I\subset\mathbb{H}$ is an isotropic space, then
	$\mu_{\mathbb{H}\otimes V}$ acts trivially on
	\begin{align*}
	  \psi_I:=\sum_{F\in\Hom(X\to I)} \delta_{F} \in L^2(\Hom(X\to\mathbb{H})).
	\end{align*}
  \end{enumerate}
\end{lemma}

The second part of the lemma makes a connection between rank-deficient subrepresentations and isotropic spaces.
Generalizations of this will be the central theme in the rest of this work.

The lemma is most easily proved by invoking the \emph{Weyl representation} of the Heisenberg group
introduced in Sec.~\ref{sec:oscillator}.

\begin{proof}
  By Eq.~\eqref{eqn:symplectic action}, the adjoint representation $\mathrm{Ad}_{\mu_V}: A\mapsto \mu_V A \mu_V^{\dagger}$ on $\End(L^2(X^*))$ permutes the Weyl operators $\{ W_V(v) \}_{v\in V}$ and is thus isomorphic to the permutation representation $\pi$.
  But by Corollary~\ref{cor:full factorization},
  \begin{align*}
	\mu_{\mathbb{H}\otimes V}
	\simeq
	\mu_V \otimes \mu_V^{(d(\mathbb{H}))}
	=
	\mu_V \otimes \bar\mu_V
	\simeq
	\mathrm{Ad}_{\mu_V}.
  \end{align*}
  This proves the first claim.

  Next, note that the adjoint representation acts trivially on $W_V(0)=\Id$.
  Our strategy is to show that for every isotropic space $I\subset \mathbb{H}$, one can choose the isomorphisms employed in the first part, to map $W_V(0)$ to $\psi_I$.
  Indeed, the isomorphism $\mathrm{Ad}_{\mu_V}\simeq \mu_V\otimes \bar\mu_V$ is implemented by
  \begin{align*}
	  i_1: \End(L^2(X^*))
	  \to
	  L^2(X^*)^{\otimes 2},
	  \qquad
	  \delta_y\otimes\delta_{y'}^T
	  \mapsto
	  \delta_y\otimes \delta_{y'},
  \end{align*}
  where $\delta_{y'}^T$ is the map acting on $\psi\in L^2(X^*)$ as $\psi\mapsto \psi(y')$.
  Choose an orthogonal basis $\{f_1, f_2\}\subset\mathbb{H}$ as in Eq.~\eqref{eqn:discriminant} and let $i_2$ be the associated isomorphism defined in Corollary~\ref{cor:full factorization}.
  Then
  \begin{align*}
	W_V(0)
	= \Id_V
	= \sum_{y \in X^*} \delta_y\otimes \delta_y^T
	\stackrel{i_1}{\mapsto}
	\sum_{y \in X^*} \delta_y\otimes \delta_y
	  \stackrel{i_2^{-1}}{\mapsto}
	  \sum_{y \in X^*} \delta_{(f_1-f_2)\otimes y}=\psi_{I_-}
  \end{align*}
  where $I_- = \FF_q(f_1 - f_2)$ is isotropic.
	Finally, any isotropic $I\subset U$ can be written this way, with a suitable choice of 
  orthogonal basis $\{f_1, f_2\}$ and associated isomorphism $i_2$.
\end{proof}

\subsection{Quotient spaces and self-orthogonal Calderbank-Shor-Steane codes}
\label{sec:css}

In this section, we will introduce the type of spaces that will turn out to contain all rank-deficient representations.
In the field of quantum error correction, these spaces are called (tensor powers of) \emph{self-orthogonal Calderbank-Shor-Steane (CSS) codes}
\cite{steane1996error,calderbank1996good,steane1996multiple}.

\begin{definition}\label{def:single css}
  Let $N\subset U$ be an isotropic space.
  The \emph{self-orthogonal CSS code} associated with $N$ is the space
  \begin{align*}
	\{ \Phi \in L^2(U) \;|\; \supp \Phi \subset N^\perp,  \Phi(u)=\Phi(u')\quad\forall\,u-u' \in N \}
  \end{align*}
  of functions whose support is contained in $N^\perp$ and which are constant on cosets of $N$.
\end{definition}

We will require an extension of this definition to functions on the tensor product space  $U\otimes X^*\simeq \Hom(X\to U)$.

\begin{definition}\label{def:css}
  Let $N\subset U$ be an isotropic space.
  The \emph{tensor power CSS code} associated with $N$ is the subspace $C_N\subset L^2(\Hom(X \to U))$
  of all functions $\Phi$ satisfying
  \begin{align}\label{eqn:css code}
	   \begin{cases}
       \Phi(F) = \Phi(F'),\quad\text{ if } F-F' \in \Hom(X \to N),\\
       \supp \Phi \subseteq \Hom(X\to N^\perp).
     \end{cases}
  \end{align}
\end{definition}

Using Lemma~\ref{lem:factorization general}, one can see that the codes defined above are indeed tensor powers of the self-orthogonal CSS codes of Definition~\ref{def:single css}.
We also note that projectors onto tensor powers of CSS codes have previously been identified in the commutant of the Clifford group \cite{nebe-book, zhu2015multiqubit, grossnezamiwalter}.

Tensor power CSS codes carry a representation of $\Sp(V)$ that is associated with the orthogonal space $N^\perp / N$:

\begin{lemma}\label{lem:css representation}
  Let $N\subset U$ be an isotropic space.

  The quotient space $U'=N^\perp/N$ inherits an orthogonal form with dimension and discriminant given by, respectively
  \begin{align*}
	\dim U' = U - 2\dim N,
	\qquad
	d(U')=(-1)^{\dim N} d(U).
  \end{align*}
  The stabilizer group $O_N\subset O(U)$ of $N$ acts on $U'$.
  The maps that arise this way are exactly $O(U')$.

  The restriction of $\mu_{U\otimes V}$ to $O_N\times \Sp(V)$ acts on $C_N$.
  As a representation of $O(U')\times \Sp(V)$, it is equivalent to $\mu_{U'\otimes V}$.
\end{lemma}

In view of this Lemma, we will say that a tensor power CSS code $C_N$ has \emph{rank} $r$, if it carries a rank-$r$ represetation, or, equivalently, if $\dim N = (t-r)/2$.

\begin{proof}
  Let $\{u_1, \dots, u_k\}$ be a basis of $N$.
  There exist $u_1', \dots, u_k'$ such that
  \begin{align}\label{eqn:system of equations}
	\beta(u_i, u_j') = \delta_{i,j}
  \end{align}
  (because, for each $j$, Eq.~\eqref{eqn:system of equations} is an underdetermined system of linear equations for $u_j'$).
  Then $\mathbb{H}_i = \langle u_i, u_i' \rangle$ is a hyperbolic plane, and we arrive at an orthogonal decomposition
  \begin{align}\label{eqn:nperp decomposition}
	U = \mathbb{H}_1 \oplus\dots\oplus \mathbb{H}_k \oplus U'=:H\oplus U',
  \end{align}
  where $U'=H^\perp$ is the orthogonal complement of the hyperbolic planes.
  Equation~\eqref{eqn:nperp decomposition} implies:
  (1) The discriminant of $U'$ is $d(U')=(-1)^k d(U)$,
  and
  (2)
  the orthogonal complement $N^\perp$ equals $N \oplus U'$, and we thus have $N^\perp/N\simeq U'$.
  Because the form on $U'$ is inherited from the one of $U$, it is clear that $O_N$ acts isometrically on $U$.
  Let $i: O_N \to O(U')$ be the homomorphism that maps elements of $O_N$ to their action on $O(U')$.
  Then $i$ is onto: If $g\in O(U')$, then, using the decomposition~\eqref{eqn:nperp decomposition}, we can embed $g$ as $\id\oplus g$ into $O_N$.

  By Corollary~\ref{cor:factorization U},
  \begin{align*}
	L^2(\Hom(X\to U)) &\simeq
	L^2(\Hom(X\to H))
	\otimes L^2(\Hom(X\to U')),\\
	\mu_{U\otimes V} &\simeq 
  \mu_{H\otimes V}\otimes\mu_{U'\otimes V}.
  \end{align*}
  By Lemma~\ref{lem:hyperbolic plane}, $\mu_{H\otimes V}$ acts trivially on
  \begin{align*}
	\psi_{u_1}\otimes \dots \otimes \psi_{u_k}
    &=\sum_{y_1, \dots, y_k \in X^*} \delta_{u_1\otimes y_1} \otimes \dots \otimes \delta_{u_k\otimes y_k}\\
	&\simeq
	\sum_{F \in \Hom(X\to N)} \delta_F
	\quad\in L^2(\Hom(X\to H)).
  \end{align*}
  Thus
  \begin{align*}
	C_N
	\simeq
	\Big(
	  \sum_{F \in \Hom(X\to N)} \delta_F
	\Big) \otimes L^2(\Hom(X\to U')),
  \end{align*}
  on which $\mu_{U\otimes V}$ acts as $\mu_{U'\otimes V}$.
\end{proof}

In the remainder of this section, we introduce two concepts that will be used in Section~\ref{sec:main} to reconstruct the codes a rank-deficient representation lives on.

A natural orthogonal basis on a tensor power CSS code is given by \emph{coset states}
(the generalization of Eq.~\eqref{eqn:coset powers}).
Given an isotropic subspace $N\subset U$, an $F\in\Hom(X\to N^\perp)$, and a coset
\begin{align*}
  [F]\in\Hom(X\to N^\perp)/\Hom(X\to N)\simeq \Hom(X\to N^\perp / N),
\end{align*}
the associated \emph{tensor power coset state} is
\begin{align*}
  e_{[F]} = \sum_{G\in [F]} \delta_{G} \in C_N.
\end{align*}
We will occassionally write $[F]_N$, if the vector space $N$ is not unambiguously clear from context.
The set
\begin{align}\label{eqn:cn basis}
  \{ e_{[F]} \;|\; [F]\in\Hom(X\to N^\perp / N)\}
\end{align}
is an orthogonal basis for $C_N$.
Note that if $[F]=[F']$, then $F=F'+\Delta$ for some $\Delta\in\Hom(X\to N)$ and thus
\begin{align}\label{eqn:same weight}
  (F')^T F'
  = F^T F +F^T\Delta + \Delta^T F + \Delta^T \Delta = F^T F.
\end{align}
In particular, $e_{[F]}$ carries the $\mathcal{N}$-weight $B=F^T F$.

With each $F\in\Hom(X\to U)$, we associate the isotropic space
\begin{align}\label{eqn:NF}
  N_F = \range F \cap (\range F)^\perp,
\end{align}
which is the radical of the range of $F$.

\begin{lemma}\label{lem:dimension bound}
  Let $F\in\Hom(X\to U)$ be such that $\rank F^T F=r$.
  Then we have the dimension bound
  \begin{align}
		\label{eqn:nf bound}
		\dim N_F \leq \lfloor (t-r)/2 \rfloor.
  \end{align}
\end{lemma}

\begin{proof}
  We decompose $U$ as $U_1 \oplus U_2 \oplus U_3$, where $U_1=N_F$, $U_2$ is a complement to $N_F$ in $\range F$,  and $U_3$ a complement to $\range F$ in $U$ (c.f.\ Fig.~\ref{fig:spaces}).
  By construction,  the space $U_2$ is non-degenerate and of dimenesion $r$, which implies that $U_2^\perp$
  is $(t-r)$-dimensional and non-degenerate.
  Thus $N_F=U_1\subset U_2^\perp$ is isotropic and contained in a $(t-r)$-dimensional non-degenerate space, which implies by Eq.~\eqref{eqn:isotropic dimension bound} that $\dim N_F\leq (t-r)/2$.
\end{proof}

\begin{figure}
	\includegraphics{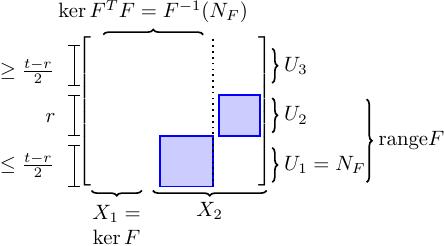} 
	\caption{\label{fig:spaces}
	Illustration of the various subspaces we will associate with an $F\in\Hom(X\to U)$.
	In Lemma~\ref{lem:genesis} and in the proof of the Main Theorem, the domain $X$ will be decomposed as a direct sum of $X_1=\ker F$ and some complement $X_2$.
	In Lemma~\ref{lem:dimension bound} and in the proof of the Main Theorem, we decompose $U$ as a direct sum of $U_1=N_F=\range F \cap (\range F)^\perp$;
	$U_2$, some complement of $U_1$ within $\range F$; and $U_3$, some complement of $\range F$.
	These choices decompose $\Hom(X\to U)$ into six different subspaces $\Hom(X_i\to U_j)$, each of which can be visualized as a block in the matrix depicted.
	In Lemma~\ref{lem:genesis}, the map $\Delta$ lives in the lower left hand side block, $\Hom(X_1\to U_1=N_F)$.
	In Lemma~\ref{lem:trolling}, we extend this to elements $\Delta=F-F'$ of the entire lower block $\Hom(X\to U_1)$, subject to a rank constraint.
	In the proof of the Main Theorem, $G$ lives in the left block $\Hom(X_1\to U)$.
	One could further subdivide $X_2$ into $X_2\cap F^{-1}(N_F)$ (left side of the dotted line), and some complement (right side of the dotted line).
	We do not make use of this division in our argument.
	With respect to this choice, $F$ is non-zero exactly on the two shaded blocks (where, in fact, it is invertible).
	}
\end{figure}

\subsection{Fourier transforms}
\label{sec:fourier}

Central to the proof of our main result will be the fact that subrepresentations of the oscillator representation are closed under certain Fourier transforms.
By a \emph{Fourier transform}, we mean a map of the form $\mu(J_B)$ defined in Eq.~\eqref{eqn:J oscillator}, for $B: X^* \to X$ symmetric and invertible.

A standard result from harmonic analysis says that the support of a function is contained in a vector space if and only if its Fourier transform is supported on a (suitably defined) orthogonal complement.
The version we will require below reads:

\begin{lemma}\label{lem:ft invariance lemma}
 Let $B: X^*\to X$ be symmetric and invertible,
 let $\Phi\in L^2(\Hom(X\to U))$, and
 let $U'\subset U$ be a subspace.

 Then the support of $\Phi$ is contained in the space $\Hom(X\to U')$
 if and only if the support of the Fourier transform
 $\tilde\Phi:=\mu_{U\otimes V}(J_B) \Phi$ is contained in $\Hom(X\to {U'}^\perp)$.

 What is more, $\Phi$ is the indicator function on $\Hom(X\to U')$ if and only if $\tilde\Phi$ is the indicator function on $\Hom(X\to{U'}^\perp)$.
\end{lemma}

Since the proof follows the standard template for such results in harmonic analysis, we have deferred it to Appendix~\ref{app:ft invariance}.

Inspecting the generators in Sec.~\ref{sec:oscillator}, it is clear that only Fourier transforms -- i.e.\ generators from $\mathcal{J} \subset\Sp(V)$ -- can possibly affect the rank of an element $\Phi\in L^2(\Hom(X\to U))$.
This is the reason such maps figure prominently in our argument.
By analyzing the action of Fourier transforms, one can easily derive further statements about the ``rank spectrum'' of representation spaces.
The following proposition is one such example.

\begin{proposition}\label{prop:rank spectrum}
  Let $(\mathcal{K}, \rho_\mathcal{K})$ be an irreducible $\Sp(V)$-subrepresentation of 
  $\mu_{U\otimes V}$, where $\mathcal{K}\subset\mathrm{Hom}(X\to U)$. Let
  \begin{align*}
	   R = \{ \rank B \;|\; B
     \text{ is a weight that appears in $\rho_\mathcal{K}|_\mathcal{N}$} 
     \}
  \end{align*}
  be the set of values the rank takes on the $\mathcal{N}$-spectrum of the representation.
  Then $R$ is a contiguous range of integers.
\end{proposition}

As the rest of the arugment will not rely on Proposition~\ref{prop:rank spectrum}, its proof is given in Appendix~\ref{sec:rank spectrum proof}.

\section{The classification of rank-deficient subrepresentations}
\label{sec:main}

\subsection{Informal outline of the main proof}

Let $\mathcal{K}\subset L^2(\Hom(X\to U))$ be a representation space of rank $r<t$.
We aim to show that there is some $\Phi\in\rep$ can be written as a linear combination
\begin{align}\label{eqn:goal}
  \Phi=\sum_{\substack{N\text{ isotropic}
  }}
  \Phi_N,
\end{align}
of components $\Phi_N$ in suitable tensor power CSS code spaces $C_N$. This, together with Lemma~\ref{lem:css representation}, will imply the Main Theorem.

One of the defining properties of elements $\Phi_N$ of $C_N$ is that they are constant on cosets of $\Hom(X\to N)$.
It is not obvious how one can derive such \emph{invariance properties} from \emph{rank deficiency}.

To achieve this, we rely on the fact that $\mathcal{K}$ is closed under certain Fourier transforms.
More precisely, if we decompose $X$ as a direct sum $X_1\oplus X_2$, then any $F: X\to U$ can be written as the sum of two blocks $F=F_1+F_2$ with $F_i: X_i \to U$ (Fig.~\ref{fig:branches}).
Now fix some $F_2$ and consider the dependency $\phi: F_1 \mapsto \Phi(F_1+F_2)$ of $\Phi$ on the first block alone.
It turns out that rank deficiency imposes linear constraints on the maps $F_1$ that can appear in the support of $\phi$.
But, as we have recalled in Sec.~\ref{sec:fourier}, if the support of a function is contained in a linear subspace, then its Fourier transform is invariant under translations along the orthogonal complement.
Closure of  $\mathcal{K}$ under Fourier transforms then implies invariances of the type that occur in CSS codes for any $\Phi\in\mathcal{K}$.
This first step of recovering a CSS code structure is made precise in Lemma~\ref{lem:genesis}.

\begin{figure}
	\includegraphics{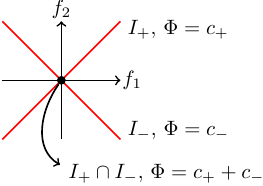} 
    \caption{
	  \label{fig:branches}
	  ``Branch and stem'' structure of rank-$0$ subrepresentations of $\mu_{\mathbb{H}\otimes V}$ associated with the hyerpbolic plane.
	  The vectors $f_1, f_2$ denote an orthogonal basis of $\mathbb{H}$.
	  The red lines $I_\pm$ are the two isotropic spaces.
	  A rank-deficient subrepresentation (Eq.~\eqref{eqn:rank def intuition}) of $\mu_{\mathbb{H},V}$ takes values that are constant on the ``branches''
	  $\{ F\;|\; N_F = I_\pm \}$, while the values add up on the ``stem'' $\{0\}$ where the spaces intersect.
    }
\end{figure}

The next challenge we are facing is that $\Phi$ is a linear combination of elements from \emph{different} codes,
so that there is no single space $N$ under which $\Phi$ is invariant.
Indeed, the symmetries found in the first step are only ``local'' in that they depend on the fixed block $F_2$.
To get some feeling for what we can expect, we look at the simplest non-trivial example: $\mu_{\mathbb{H}\otimes V}$ with $\mathbb{H}$ the hyperbolic plane.

The plane has a orthogonal basis $\{f_1, f_2\}$, with
\begin{align*}
  \beta(f_1,f_1)=1, \qquad \beta(f_2,f_2)=d(\mathbb{H})=-1.
\end{align*}
There are two isotropic spaces, $I_\pm = \FF_q\,(f_1 \pm f_2)$ (c.f.\ Fig~\ref{fig:branches}).
It follows that there are two tensor power CSS codes $C_{I_\pm}$ in $\Hom(X\to \mathbb{H})$.
They are one-dimensional, proportional to the vectors $\psi_{I_\pm}$ defined in Lemma~\ref{lem:hyperbolic plane}.
Thus, for $c_\pm\in\CC$, the vector
\begin{align}\label{eqn:rank def intuition}
  \Phi=c_+ \psi_{I_+} + c_- \psi_{I_-} \in L^2(\Hom(X\to\mathbb{H}))
\end{align}
carries a rank-$0$ representation (and we will see that these are the only rank-deficient subrepresentation of $\mu_{\mathbb{H}\otimes V}$).
Using Lemma~\ref{lem:hyperbolic plane}
\begin{align*}
  \Phi(F)=
  \left\{
	\begin{array}{ll}
	  c_+ & \rank F^TF = 0, N_F = I_+ \\ 
	  c_- & \rank F^TF = 0, N_F = I_- \\
	  c_+ + c_-\qquad& \rank F^T F = 0, N_F = \{0\} \\ 
	  0 & \rank F^T F = 1.
	\end{array}
  \right.,
\end{align*}
a situation sketched in Fig.~\ref{fig:branches}.
Embracing a horticultural analogy, $\Phi$ is constant on the two ``branches'' $\{ F\;|\; N_F = I_\pm \}$, while the values add up on the ``stem'' $\{0\}$, where the spaces intersect.

This structure generalizes to higher-dimensional orthogonal spaces $U$.
Define the ``generalized branches'' to be
\begin{align*}
  B_N := \{ F\in \Hom(X\to U)\;|\; \rank F^T F = r, N_F = N \}.
\end{align*}
Then Lemma~\ref{lem:trolling} states that on each $B_N$, a vector $\Phi$ in a rank-deficient representation exhibits the invariance under $\Hom(X\to N)$ that is characteristic of elements of the code $C_N$.
More precisely:
\begin{align*}
  F, F' \in B_N, \quad (F-F')\in\Hom(X\to N) \qquad\Rightarrow\qquad \Phi(F)=\Phi(F').
\end{align*}
Thus $\Phi$ is well-defined on sets $B_N/\Hom(X\to N)$.

After this, we ``prune off the branches'' by setting
\begin{align*}
  \Phi':=\Phi-\sum_{\substack{N\text{ isotropic}\\\dim N=\lfloor(t-r)/2\rfloor}}\,\sum_{[F]\in B_N/\Hom(X\to N)} \Phi([F])\, e_{[F]}.
\end{align*}
The right-hand summand involves the coset states $e_{[F]}$, which are elements of the respective code $C_N$.
The support of the remainder $\Phi'$ is thus contained in the ``stem''.
We conclude the argument by showing that representations with rank $< t$ do not contain
non-zero vectors supported on such a stem, so in fact $\Phi'=0$.

This final step again relies on Fourier transforms.
Roughly, the ``stem'' is a ``small'' space, so that by the uncertainty principle, Fourier 
transforms will have ``large'' support -- so large, in fact, that they are guaranteed to 
contain higher-rank elements.

\subsection{Proof of the Main Theorem}

\begin{lemma}\label{lem:genesis}
  Let $\rep\subset\mathrm{Hom}(X\to U)$ be a subrepresentation of rank $r<t$.
  Let $\Phi\in\rep$ and $F\in\supp\Phi$ such that $\rank F^T F =r$, and let $N_F$ be as
  in eqn.~\eqref{eqn:NF}. If $\Delta\in\Hom(X\to N_F)$ is such that
  \begin{align}\label{eqn:full rank}
	\range F_{|\ker \Delta}= \range F,
  \end{align}
  then
  \begin{align*}
	\Phi(F)=\Phi(F+\Delta).
  \end{align*}
\end{lemma}

\begin{proof}
  Set $X_1 = \ker F$.
  The assumption~\eqref{eqn:full rank} implies that there is a complement $X_2$ of $X_1$ contained in $\ker\Delta$.
  This choice induces a decomposition $\Hom(X\to U)=\Hom(X_1\to U)\oplus\Hom(X_2\to U)$ with $\Delta\in\Hom(X_1\to U)$, $F\in\Hom(X_2\to U)$.

  Let $B: X_1^* \to X_1$ be invertible, let $\mu_{U\otimes V_1}(J_B)$ be the associated Fourier transform, and let $i$ be the isomorphism~\eqref{eqn:isomorphism factorization}.
  By Section~\ref{sec:fourier} and Lemma~\ref{lem:factorization general}, the vector
	\begin{align*}
		\tilde \Phi:=\big(\,i^{-1}\,(\mu_{U\otimes V_1}(J_B)\otimes \mu_{U\otimes V_2}(\Id))\, i\,\big) \,\Phi
	\end{align*}
  is an element of $\rep$.
  Thus, by the assumption on the rank of the representation, $\tilde\Phi$ has support only on maps $F'\in\Hom(X\to U)$ with $\rank (F')^T F' \leq r$.

  If $F'=G+F$ for some $G\in\Hom(X_1\to U)$, then
  \begin{align*}
    \range (G+F)
	=
    \langle \range G\cup\range F\rangle.
  \end{align*}
  The condition
  $\rank(G+F)^T(G+F) \leq r$
  is equivalent to demanding that $\range (G+F)$ has rank at most $r$ as an orthogonal space.
  This implies
  \begin{align}\label{eqn:in nf}
	\range G\subset \langle \range F \cup (\range F)^\perp \rangle = N_F^\perp.
  \end{align}

  Set
  \begin{align*}
	\phi \in L^2(X_1\to U),
	\qquad
	\phi(G) = \Phi(G+F),
	\qquad
	\tilde\phi = \mu_{U\otimes V_1}(J_B)\phi.
  \end{align*}
  Then
  \begin{align*}
	\tilde\Phi(G+F)
	=
	\tilde \phi(G),
  \end{align*}
  so that the preceding discussion implies that $\supp\tilde\phi\subset \Hom(X\to N_F^\perp)$.
  Thus Lemma~\ref{lem:ft invariance lemma} implies that $\phi$ is constant on cosets of
  $\Hom(X_1\to N_F)$, a space which includes $\Delta$.
\end{proof}

The next lemma extends the invariances -- essentially by using the fact that there is a some freedom in choosing the complement $X_2$ to $\ker F$ that appears in the proof above.

\begin{lemma}\label{lem:trolling}
  Let $\rep$ be a representation of rank $r<t$, and $\Phi\in\rep$.
  Let $N\subset U$ be an isotropic space, and set
  \begin{align*}
	B_N := \{ F\in \Hom(X\to U)\;|\; \rank F^T F = r, N_F = N \}.
  \end{align*}
  Then on $B_N$, $\Phi$ is invariant under $\Hom(X\to N)$:
  \begin{align}\label{eqn:trolling f f'}
	\Phi(F) = \Phi(F')\qquad \forall\,F, F' \in B_N, (F-F')\in\Hom(X\to N).
  \end{align}
\end{lemma}

The proof uses the \emph{probabilistic method} \cite{alon2004probabilistic}: The strategy is to ascertain the (deterministic) exsistence of an object by showing that a randomized construction yields one with positive probability.
Presumably an explicit construction would offer us more insight into the structure of the problem.
We leave such a \emph{derandomization} for future work.

\begin{proof}
  Let $F, F'$ be as in Eq.~\eqref{eqn:trolling f f'}.
  The aim is to show that there exists a ``mid-point'' $G$ such that both $F$ with $\Delta=(G-F)$, as well as $F'$ with $\Delta'=(F'-G)$ fulfill the assumptions of Lemma~\ref{lem:genesis}.
  It then follows that $\Phi(F)=\Phi(G)=\Phi(F')$.

  We claim that if $\Delta$ is chosen uniformly at random from $\Hom(X\to N)$,
  then, with probability strictly larger than $1-\frac1{q-1}$, it holds that
  $\range F_{|\ker\Delta}=\range F$, i.e.\ Lemma~\ref{lem:genesis} applies to $F, \Delta$.

  Before turning to the analysis of the randomized procedure, we state two preparatory facts.
  First, for each subspace $Z\subset X$, it holds that
  \begin{align}\label{eqn:trolling criterion}
	\range F_{|Z} = \range F
	\qquad\Leftrightarrow\qquad
	\dim Z -\dim (Z \cap \ker F) = \rank F.
  \end{align}
  Second, for each $\Delta\in\Hom(X\to N)$, Lemma~\ref{lem:dimension bound} gives the dimension bound
  \begin{align}\label{eqn:rank bound}
    \dim \ker \Delta
	\geq n - \dim N 
	\geq t - \dim N 
	\geq t - (t-r)/2 
	= r + (t-r)/2
    \geq \rank F. 
  \end{align}

  Now assume $\Delta$ is distributed uniformly at random.
  From the previous equation, any $\rank F$-dimensional subspace $Z$ will occur within $\ker\Delta$ with equal probability.
  By Eq.~\eqref{eqn:trolling criterion}, if $\dim(Z\cap \ker F)=0$ for some such $Z$, then the assumption of Lemma~\ref{lem:genesis} is met.

  There are $(q^k-1)/(q-1)$ one-dimensional spaces in a $k$-dimensional vector space.
  Thus, the probability that any fixed one-dimesional subspace is contained in a randomly chosen $z$-dimensional one is $(q^z-1)/(q^n-1)$.
  By the union bound, the probability that at least one element of a fixed $(n-z)$-dimensional space is contained in a $z$-dimensional random one is therefore upper-bounded by
  \begin{align*}
	\frac{q^{n-z}-1}{q-1}\frac{q^{z}-1}{q^n-1}
	=
	\frac1{q-1}
	\frac{(q^{n}-q^{z}-q^{n-z}+1)}{q^n-1}
	<
	\frac1{q-1} \quad (\forall\,z\leq t).
  \end{align*}
  This establishes the claim made at the beginning of the proof.

  Now set $G=F+\Delta$.
  The distribution of $\Delta'=F'-G = (F'-F)-\Delta$ is the same as the distribution of $\Delta$.
  Thus Lemma~\ref{lem:genesis} applies to $F', \Delta'$ with the same probability.

  We conlcude by the union bound that the probability of the construction working in both cases simultaneously is strictly larger than $1-\frac2{q-1}\geq 0$.
\end{proof}

\begin{proof}[Proof (of the Main Theorem)]

  Let $\Phi\in\rep$ carry an $\mcal{N}$-weight $B$ of rank $r$.
  By Lemma~\ref{lem:genesis} and Lemma~\ref{lem:trolling}, $\Phi$ is well-defined on cosets $B_N/\Hom(X\to N)$.
  Set
  \begin{align}
	\label{eqn:cropped}
    \Phi':=\Phi-\sum_{\substack{N\text{ isotropic}\\\dim N=\lfloor(t-r)/2\rfloor}}\,\sum_{[F]\in B_N/\Hom(X\to N)} \Phi([F])\, e_{[F]}.
  \end{align}

  We will prove that $\Phi'$ is actually equal to zero, using a Fourier-transform argument as in Lemma~\ref{lem:genesis}.

  For the sake of reaching a contradiction, assume that $\Phi'\neq 0$ and choose an $F\in\supp \Phi'$ such that
  \begin{align}\label{eqn:main proof max rank}
	\rank F = \max_{F'\in\supp \Phi'} \rank F'.
  \end{align}

  As in the proof of Lemma~\ref{lem:genesis}, set $X_1=\ker F$, choose some complement $X_2$ to $X_1$,
  an invertible symmetric $B: X_1^*\to X_1$,
  set $V_i=X_i\oplus X_i^*$,
  and define $\phi', \tilde\phi'\in L^2(X_1\to U)$ as
	\begin{align*}
		\phi'(G) := \Phi'(F+G),
		\qquad
		\tilde\phi' := \mu_{U\otimes V_1}(J_B)\phi'.
	\end{align*}
  Then,  with $\tilde\Phi' := (\mu_{U\otimes V_1}(J_B)\otimes \mu_{U\otimes V_2}(\Id))\Phi'  \in \mathcal{K}$, it holds that
  \begin{align}\label{eqn:main proof ft domain}
		\tilde\Phi'(F+G)=\tilde\phi'(G) \quad \forall G\in\Hom(X_1\to U).
  \end{align}

  We decompose $U$ as $U_1 \oplus U_2 \oplus U_3$, where $U_1=N_F$, $U_2$ is a complement to $N_F$ in $\range F$,  and $U_3$ a complement to $\range F$ in $U$ (c.f.\ Fig.~\ref{fig:spaces}).
  Let $G=G_1\oplus G_2\oplus G_3$, $G_i\in\Hom(X_1\to U_i)$ be an element of $\supp\phi'$.
  Because $\beta$ restricted to $U_2$ is non-degenerate, it follows from $(G+F)^T (G+F)=B$ that $G_2=0$.
  By Eq.~\eqref{eqn:main proof max rank}, $G_3=0$.
  From Lemma~\ref{lem:genesis}, $\phi'$ is invariant under $\Hom(X_1 \to N_F)=G_1$.
  Thus $\phi'$ is proportional to the indicator function on $\Hom(X_1\to N_F)$.
  Hence $\tilde\phi'$ is proportional to the indicator function on $\Hom(X_1 \to N_F^\perp)$.
  But $N_F^\perp$ contains the $r$-dimensional non-degenerate space $U_2$ and has dimension $\dim N_F^\perp > t-\lfloor (t-r)/2 \rfloor \geq r+(t-r)/2$.
  Therefore, as an orthogonal space, $N_F^\perp$ has rank strictly larger than $r$
  and hence contains a non-isotropic vector $u\not\in\range F$.
  If $G\in\Hom(X_1\to N_F^\perp)$ has $u$ in its range, then $\rank (G+F)^T (G+F) \geq r+1$.
  But by Eq.~\eqref{eqn:main proof ft domain}, $G+F$ appears in the support of $\tilde\Phi'$, contradicting the assumption that $\mathcal{K}$ has rank $r$.

  It follows that $\Phi$ is in the span of the rank-$r$ tensor power CSS codes $\mcal{C}_r$.
  If $\rep$ is irreducible, then it is spanned by the orbit $\symp(V)\cdot\Phi$.
  But Lemma \ref{lem:css representation} says that $\mcal{C}_r$ is invariant under the $\symp(V)$ action, so $\rep\subseteq\mcal{C}_r$.
  By the same lemma, $t-r$ is even.
  The Main Theorem therefore holds for irreps, and hence for all representations.
\end{proof}

\subsection{The connection to the $\eta$ correspondence}
\label{sec:eta}

Here, we will combine the respective main results of this work and of Ref.~\cite{gh17} to arrive at a complete decomposition of\\$L^2(\Hom(X\to~U))$ in terms of irreducible $\Sp(V)$ subrepresentations.

One can generate $\Sp(V)$-subrepresentations by choosing an isotropic subspace $N$ and a $\tau\in\mathrm{Irr}(O(N^\perp/N))$, and then embeddeding $\eta(\tau)$ into the code $C_N$.
Isotropic spaces of the same dimension will give rise to isomorphic $\Sp(V)$-representations.
The main observation of the next lemma is that, while in general different CSS codes may have non-trivial intersections, the representation spaces arising in the way just described  are linearly independent.
This allows us to identify the joint action of $U(O)$ and $\Sp(V)$ on their span as a certain induced representation.

Recall that by Lemma~\ref{lem:css representation}, there is a homomorphism $i: O_N \to O(N^\perp/N)$ from the stabilizer group $O_N\subset O(U)$ of an isotropic subspace onto the orthogonal group of $N^\perp/N$.
Thus, if $\tau$ is a representation of $O(N^\perp/N)$, then $\tau\circ i$ represents $O_N$.
In this section, we will implicitly make this identification and we will not distinguish notationally between $\tau$ and $\tau\circ i$.

\begin{lemma}\label{lem:induced}
  Let $N\subset U$ be an isotropic space
  and let $\tau\in\mathrm{Irr}(N^\perp/N)$.

  Let $\mathcal{K}\subset L^2(\Hom(X\to U))$ be the subspace on which $\Sp(V)$ acts as $\eta(\tau)$.
  Then, as an $O(U)\times \Sp(V)$-representation,
  \begin{align}\label{eqn:induced lemma}
	\mathcal{K}
	\simeq
	\mathrm{Ind}_{O_N}^{O(U)}(\tau)
	\otimes
	\eta(\tau).
  \end{align}
\end{lemma}

\begin{proof}
  Set $U'=N^\perp/N$.
  By Lemma~\ref{lem:css representation} and Theorem~\ref{thm:gurevich}, there is a unique $O(U')\times \Sp(V)$-representation space $\mathcal{K}_{0}$ of type $\tau\otimes\eta(\tau)$ in $C_N$.

  The \emph{isotropic Grassmanian}
  \begin{align*}
	I_r = \{ N \;|\; N\text{ isotropic}, \dim N = (t-r)/2 \} \simeq O(U)/O_{N}
  \end{align*}
  can be identified with the cosets $O(U)/O_N$.
  Let $\{ g_i \}_{i=1}^{|I_r|}$ be a choice of representatives for each coset.
  Define
  \begin{align*}
	\mathcal{K}_{[g_i]}
	= \mu_{U\otimes V}(g_i)\,\big(\mathcal{K}_{0}\big).
  \end{align*}
  As $\Sp(V)$-representation spaces, the $\mathcal{K}_{[g_i]}$ are all equivalent to $\eta(\tau)$.
  Conversely, from Theorem~\ref{thm:main}, every $\Sp(V)$-representation of type $\eta(\tau)$ is contained in their span.
  Therefore,
  \begin{align*}
	\mathcal{K} = \mathrm{span} \{ \mathcal{K}_{[g_i]} \}_{i=1}^{|I_r|}.
  \end{align*}

  We claim that the	spaces $\mathcal{K}_{[g_i]}$ are linearly independent.

  Indeed: 
  We need to show that for each $i$, the sapce $\mathcal{K}_{[g_i]}$ intersects the span $\mathcal{K}'$ of the other spaces only at $\{0\}$.
  Since $O(U)$ acts transitively on the $\mathcal{K}_{[g_i]}$, it is enough to treat the case $i=1$.
  As
  $\mathcal{K}'$ and $\mathcal{K}_{[g_1]}$
  are
  $O_N\times \Sp(V)$
  representation spaces, and because
  $\mathcal{K}_{[g_1]}$
  is irreducible,
  we have the alternatives
  \begin{align*}
	\mathcal{K}_{[g_1]} \subset \mathcal{K}'
	\qquad
	\text{ or }
	\qquad
	\mathcal{K}_{[g_1]} \cap \mathcal{K}' = \{0 \}.
  \end{align*}
  It thus suffices to show that
  $\mathcal{K}_{[g_1]}$
  contains one vector that is not an element of $\mathcal{K}'$.
  Let $\Phi_1\in \mathcal{K}_{[g_1]}$ carry an $\mathcal{N}$-weight $B$ of rank $r$, let $F\in\supp\Phi_1$.
  There is some $F'\in [F]_N$ that is \emph{maximal} in the sense $\range F' = N^\perp$ (rather than its range being a strict subset of $N^\perp$).
   Since $\rank F^TF=r$, there must be some complement $W$ of $N$ in $N^\perp$ for which
  $W\subseteq \range F$. From the decomposition $\Hom(X\to N^\perp)=\Hom(X\to N)\oplus\Hom(X\to W)$ it is clear that there exists a $\Delta \in\Hom(X\to N)$ for which $F+\Delta=F'$ is maximal.
  By the invariance property of CSS codes, $F'\in\supp\Phi_1$, i.e.\ the inner product $(\delta_{F'}, \Phi_1)\neq 0$.
  In contrast, let $\Phi_i\in \mathcal{K}_{[g_1]}$ for $i\neq 1$.
  Then $F_i \in \supp\Phi_i \Rightarrow \range F_i \subset N_i^\perp$.
  But $\range F' = N_1^\perp \not\subset N_i^\perp$,
  so that $(\delta_{F'},\Phi_i)=0$.
  It follows that $\Phi_1\not\in\mathcal{K}'$, as claimed.

  The space $\mathcal{K}$ is therefore a direct sum of the
  $\mathcal{K}_{[g_i]}$.
  We will now compute the action of $O(U)$ on this direct sum.
  It suffices to consider vectors of the form
  \begin{align*}
	\mu_{U\otimes V}(g_i) (\phi\otimes\psi), 
  \end{align*}
  which span $\mathcal{K}$.
  For each $g\in O(U)$, there is a permutation $\pi \in S_{|I_r|}$ and elements $h_i\in O_N$ such that for each $g g_i = g_{\pi_i} h_i$.
  Thus
  \begin{align}\label{eqn:induced rep}
	\mu_{U\otimes V}(g)
	\,
	\big(
	  \mu_{U\otimes V}(g_i) (\phi\otimes\psi)
	\big)
	= \mu_{U\otimes V}(g_{\pi_i} h_i) (\phi\otimes\psi)
	= \mu_{U\otimes V}(g_{\pi_i})\,\big(\phi\otimes\tau(h_i)\psi\big).
  \end{align}
  But this is the action of the advertised induced representation.
\end{proof}

By Lemma~\ref{lem:css representation}, the space $N^\perp/N$ is an orthogonal space of dimension $r=t-2k$ and discriminant $d(U_r)=(-1)^k d(U)$, with $k=\dim N$.
In particular, up to orthogonal maps, $N^\perp/N$  only depends on $r$.
With this in mind, we suppress the dependency on $N$ in our notation, and define $U_r$ to be $N^\perp/N$ for \emph{some} isotropic $N$ of dimension $(t-r)/2$.
In the same vein, any two isotropic spaces of the same dimension have conjugate stabilizer groups,
and thus the isomorphism class of the induced representation in Eq.~\eqref{eqn:induced lemma} does not depend on $N$.
Again, this justifies defining $O_r$ to be $O_N$ for \emph{some} isotropic $N$ of dimension $(t-r)/2$.

Theorem~\ref{thm:gurevich}, Theorem~\ref{thm:main}, and the preceding lemma then yield the decomposition
\begin{align}\label{eqn:nice}
  \mu_{U\otimes V} \simeq
  \bigoplus_{r \in R(U)}\quad \bigoplus_{\tau \in\mrm{Irr}\,O(U_r)}\quad
  \mathrm{Ind}_{O_r}^{O(U)}(\tau)
  \otimes
  \eta(\tau),
\end{align}
with
\begin{align*}
  R(U) = \{ t- 2k \;|\; \text{ there is an isotropic } N\subset U \text{ with } \dim N = k \}.
\end{align*}
All $\Sp(V)$-irreps $\eta(\tau)$ appearing in Eq.~\eqref{eqn:nice} are indeed inequivalent: Those corresponding to different $O(U_r)$ are distinguished by their rank, whereas the inequivalence of summands of the same rank is a consequence of Theorem~\ref{thm:gurevich}.

As an $O(U_r)$-representation,
\begin{align*}
  \mathrm{Ind}_{O_r}^{O(U)}(\tau) \simeq \tau \otimes \CC^{|O(U)/O_N|}
\end{align*}
is just $\tau$ with degeneracy equal to the number of isotropic subspaces of dimension $k$.

A comparison with Theorem~\ref{thm:gurevich} shows that the $\Sp(V)$-representations in $\Theta(\tau)$ are exactly those $\eta(\tau')$, where $\tau$ appears in $\mathrm{Ind}_{O_N}^{O(U)}(\tau')$.
In terms of character inner products, and using Frobenius reciprocity:
\begin{align*}
  \langle \Theta(\tau), \eta(\tau') \rangle_{\Sp(V)}
  =
  \sum_{r\in R(U)} \langle \tau, \mathrm{Ind}_{O_r}^{O(U)}(\tau')\rangle_{O(U)}
  =
  \sum_{r\in R(U)} \langle \mathrm{Res}_{O_r}^{O(U)}(\tau), \tau'\rangle_{O_r}.
\end{align*}
As an example, we consider the case where $\tau=\id_{O(U)}$ is the trivial representation of $O(U)$.
Then
\begin{align*}
  \langle \mathrm{Res}_{O_r}^{O(U)}(\id_{O(U)}) , \tau' \rangle_{O_r}
  =
  \langle \id_{O_r} , \tau' \rangle_{O_r}
  =
  \delta_{\id_{O(U_r)}, \tau'}.
\end{align*}
Therefore,
\begin{align*}
  \Theta(\id_U) = \bigoplus_{r\in R(U)}
  \eta(\id_{O(U_r)}).
\end{align*}
has a number of components equal to the isotropy index of $U$.

\subsection{A non-CSS type rank-deficient subrepresentation}
\label{sec:counterexample}

Our main theorem makes statements only in the regime $t\leq n$.
Here, we show that it indeed cannot be extended to all pairs $t,n$.
To this end, we construct a rank-$0$ subrepresentation of $\mu_{\FF_p^3\otimes \FF_p}$, i.e.\ for the case of $t=3$ and $n=1$.
Here, $p$ is an arbitrary odd prime.
This is incompatible with Theorem~\ref{thm:main}, which posits that $t-r$ be even.
Thus, more general subrepresentations can occur for $t>n$.

Set $V=\FF_p\oplus \FF_p^*$ and $U=\FF_p^3$ with the standard orthogonal form $\beta$.
The oscillator representation $\mu_{U\otimes V}$ thus acts on $L^2(U\otimes \FF_p^*)\simeq L^2(U)$.

Our construction depends\footnote{Numerically, it appears that the resulting representation space is actually independent of the choice of $x_0$.
Numerical investigations also indicate that when substituting $U=\FF_p^3$ (which has discriminant $d(U)=1$) by a three-dimensional $U'$ with discriminant $d(U')$ a non-square, then $\mu_{U'\otimes V}$ will still act trivially on $\psi$ if $p=3$. On the other hand, for $p=5,7,11,13$, it holds that $\mu_{U'\otimes V}$ does not afford \emph{any} trivial representation space.
We will neither use, nor attempt to prove, these statements.}
on the choice of an isotropic vector $x_0\in U$.
Define $\psi\in L^2(U)$ by
\begin{align}
  \psi(z)
  =
  \left\{
	\begin{array}{ll}
	  0 \qquad& \beta(z,z)\neq 0 \text{ or } z=0 \\
	  \ell_{\beta(x_0, z)} & \beta(z,z)=0, z \neq \lambda x_0 \\
	  \ell_{2\lambda} & z = \lambda x_0.
	\end{array}
  \right.
  \label{eqn:noncss}
\end{align}
In particular, $\psi$ is supported on the set of isotropic vectors in $U$, and restricts to a Legendre symbol on every ray.

\begin{proposition}
  The representation $\mu_{U\otimes V}$ acts trivially on $\psi$.
\end{proposition}

\begin{proof}
  From the explicit definitions in Section~\ref{sec:oscillator}, one can easily see that $\psi$ affords trivial actions by the subgroups $\mathcal{N}$ (using isotropy of the support) and $\mathcal{D}$ (using the multiplicativity of the Legendre symbol).
  All elements $J_B$ of the subgroup $\mathcal{J}$ can be written as a product of an element from $\mathcal{D}$ with $J_{\mathrm{id}}$, where $\mathrm{id}: X^* \to X$ is the canonical identification of $\FF_p^*$ with $\FF_p$.
  It therefore remains to be shown that $\psi$ is stabilized by $\mu_{U\otimes V}(J_{\mathrm{id}})$.

  We begin by deriving a more convenient expression for $\psi$.
  The standard form in $\FF_p^3$ is isomorphic to $\mathbb{H}\oplus \langle -1 \rangle$.
  In other words, there exists a basis with respect to which the standard form on $\FF_p^3$ is
  \begin{align*}
	\beta(x,y) = x_1 y_2 + x_2 y_1 - x_3y_3.
  \end{align*}
  In this basis, define
  \begin{align*}
	x_a &= (1,2^{-1}a^2,a), \qquad a \in \FF_p \\
	x_\infty &= (0,2,0).
  \end{align*}
  By enumerating all points in projective space $\FF_p^3/\FF_p$, one may easily convince 
  oneself that every isotropic vector in $\FF_p^3$ is a multiple of exactly one $x_a$, for 
  $a\in\bar\FF_p := (\FF_p \cup \infty)$.
  We can choose the basis change such that the vector $x_0$ that appers in~\eqref{eqn:noncss} 
  is mapped to the vector $x_0$ as defined here.

  For $a\neq b\in \FF_p$,
  \begin{align*}
	\beta(x_a, x_b) &= 2^{-1}(a^2+b^2)-ab = 2^{-1} (a-b)^2, \\
	\beta(x_a, x_\infty) &= 2,
  \end{align*}
  so that the Legendre symbol of the inner products is constant:
  \begin{align*}
	\ell_{\beta(x_a, x_b)}=\ell_2 \qquad \forall a\neq b \in\bar\FF_p^3.
  \end{align*}
  With these definitions, $\psi$ takes a simple form:
  \begin{align*}
	\psi&=
	\sum_{a\in\bar\FF_p}
	\sum_{\lambda \in \FF_p^\times}
	\left(
	  \ell_{\beta(x_0,x_a)}
	  +
	  \ell_2
	  \delta_{a,0}
	\right)
	\ell_{\lambda}
	e_{\lambda x_a} \\
	&=
	\ell_2
	\sum_{a\in\bar\FF_p}
	\sum_{\lambda \in \FF_p^\times}
	\ell_{\lambda}
	e_{\lambda x_a}.
  \end{align*}

  We evaluate the Fourier transform
  $\tilde\psi = \mu_{U\otimes V}(J_{\mathrm{id}})\psi$
  on an isotropic vector, using Eq.~\eqref{eqn:J oscillator}:
  For $\kappa\in\FF_p^\times, b\in\bar\FF_p$, it holds that
  \begin{align*}
	\tilde\psi(\kappa x_b)
	&=
	\gamma^{-3}
	\ell_2
	\,
	\sum_{a\in\bar\FF_p}
	\sum_{\lambda \in \FF_p^\times}
	\ell_{\lambda}
	\omega(\beta(\lambda x_a, \kappa x_b)) \\
	&=
	\gamma^{-3}
	\ell_2
	\,
	\sum_{a\in\bar\FF_p, a\neq b}
	\ell_{\kappa}
	\ell_{\beta(x_a, x_b)}
	\sum_{\lambda \in \FF_p^\times}
	\ell_{\lambda}
	\omega(\lambda) \\
	&=
	\gamma^{-3}
	p
	\ell_{\kappa}
	\,
	\sum_{\lambda \in \FF_p^\times}
	\ell_{\lambda}
	\omega(\lambda) \\
	&=
	\gamma^{-2}
	p
	\ell_{\kappa}
	=
	\ell_2
	\ell_{\kappa}
	=
	\psi(\kappa x_b),
  \end{align*}
  where we have used the standard properties of quadratic Gauss sums.
  Restricted to the support of $\psi$, this is the required eigenvalue equation.

  In particular, we have found that
  $\tilde\psi$
  coincides with $\psi$ on the support of $\psi$.
  Because the oscillator representation acts isometrically, $\tilde\psi$ must thus also have the same support as $\psi$.
\end{proof}

\section{The connection to the Clifford group}
\label{sec:clifford}

The motivation for this work was to understand the appearance of projections onto CSS codes in the commutant of tensor power representations of the \emph{Clifford group} \cite{grossnezamiwalter}.
While we have opted to state our main results for representations of the symplectic group, the two cases can sometimes be precisely linked.
This is the purpose of Proposition~\ref{prop:symplectoclifford}, which will be developed in this section.

We start by recalling the basic definitions.
In addition to the oscillator representation, the Hilbert space $L^2(X^*)$ also carries a representation $W^{(m)}$ of the \emph{Heisenberg group} $H(V)$ over $V=X\oplus X^*$.
The Heisenberg group $H(V)$ is the set $\FF_q \times V$ with group law
\begin{align*}
  (\lambda, v) \circ (\lambda', v') = (\lambda+\lambda'+2^{-1}[v,v'], v+v').
\end{align*}
For $m\in\FF_q^\times$, the \emph{Weyl representation of mass $m$} on $L^2(X^*)$ is
\begin{align}\label{eqn:weyl operators}
  W_V^{(m)}({\lambda,x\oplus y})\delta_z = \omega^{(m)}(-2^{-1} y(x)+z(x)+\lambda)\,\delta_{z+y}.
\end{align}
As is true for the oscillator representation (Sec.~\ref{sec:oscillator}),
we again have that $W_V^{(-m)}$ is the complex conjugate of $W_V^{(m)}$, and again we will omit the superscript for the mass-$1$ version.
Two Weyl representations of different mass are inequivalent~\cite{gerardin}.
The Weyl and the oscillator representations are compatible
in that
\begin{align}\label{eqn:symplectic action}
  \mu_V^{(m)}(S)\,W_V^{(m)}(\lambda, v)\,\mu_V^{(m)}(S)^{-1} = W_V^{(m)}(\lambda, Sv)
\end{align}
for all $S\in\Sp(V), v\in V$.
The semi-direct product $H(V) \rtimes \symp(V)$ with automorphism
\begin{align*}
  S\,(\lambda, v)\,S^{-1} = (\lambda, Sv)
\end{align*}
is the \emph{Jacobi group} over $V$.
By Eq.~\eqref{eqn:symplectic action}, the map
\begin{align*}
  \cliff_V^{(m)}:(\lambda, v, S) \mapsto
  W_V^{(m)}(\lambda,v) \mu_V^{(m)}(S),
\end{align*}
thus
defines a representation of the Jacobi group on $L^2(X^*)$.
The operators realizing this representation form the \emph{Clifford group}.
Because the maps $\{W_V(v)\}_{v\in V}$ form a basis in $L^2(X^*)$, Eq.~\eqref{eqn:symplectic action} determines $\mu(S)$ up to a phase factor.

As $\Sp(V)$ embeds into $\Sp(U\otimes V)$ (Sec.~\ref{sec:product spaces}), so too can one embed the Heisenberg group $H(V)$ into $H(U\otimes V)$.
However, the embedding we will use is no longer canonical, but depends on the choice of a vector $u\in U$.
Roughly, we use $u$ to lift $y\in X^*$ to $u\otimes y\in U\otimes X^*$. 
Given $u$, define
\begin{align*}
  \iota_u:
  (\lambda,v)
  \mapsto
  (\beta(u,u) \lambda, u\otimes v).
\end{align*}
This is a homomorphism:
\begin{align*}
  &\iota_u\big( (\lambda,x\oplus y) \circ (\lambda',x'\oplus y') \big) \\
  =&
    \Big( \beta(u,u) (\lambda+\lambda')  2^{-1} \beta(u,u)\big(y'(x)-y(x')\big),
    u\otimes \big((x+x')\oplus (y+y')\big) \Big) \\ 
  =&
  \iota_u (\lambda,x\oplus y) \circ \iota_u (\lambda',x'\oplus y'),
\end{align*}
from which one verifies that we have a representation
\begin{align*}
  \cliff^{(u)}_{U\otimes V}: (\lambda,v, S) \mapsto
  W_{U\otimes V}(\iota_u(\lambda,v))\,\mu_{U\otimes V}(S)
\end{align*}
of the
Jacobi group over $V$ on $L^2(\Hom(X\to U))$.
It again fulfills a factorization property, generalizing Corollaries~\ref{cor:factorization U} and \ref{cor:full factorization}.

\begin{lemma}\label{lem:jacobi factorization}
 Assume $U=U_1 \oplus U_2$ is an orthogonal direct sum and
  let $u=u_1\oplus u_2$ with $u_i \in U_i$.
  Then, under the same ismorphism as introduced in Corollary~\ref{cor:factorization U},
  \begin{align*}
	\cliff^{(u)}_{U\otimes V}
	\simeq
	\cliff^{(u_1)}_{U_1\otimes V}
	\otimes
	\cliff^{(u_2)}_{U_2\otimes V}.
  \end{align*}

  If $u=\sum_{i=1}^t f_i$ for an orthogonal basis $\{f_i\}_{i=1}^t$ as in \eqref{eqn:discriminant}, then
  \begin{align*}
	\cliff^{(u)}_{U\otimes V} \simeq
	\underbrace{\cliff_V\otimes\dots\otimes\cliff_V}_{(t-1)\,\times}\otimes \cliff_V^{(d(U))}.
  \end{align*}
  In particular, if $U=\FF_q^t$ and $f_i$ is the standard orthonormal basis, then
  \begin{align*}
	\cliff^{(u)}_{\FF_q^t\otimes V}
	\simeq
	\cliff_V^{\otimes t}.
  \end{align*}
\end{lemma}

\begin{proof}[Proof (of Lemma \ref{lem:jacobi factorization})]
  The symplectic subgroup of the Clifford group factorizes according to Corollary~\ref{cor:factorization U}.
  It remains to be shown that the same is true for the image of $W(\iota_u(\lambda, v))$ under the isomorphism \eqref{eqn:orthogonal isomorphism}.
  Using $\beta(u,u)=\beta(u_1,u_1)+\beta(u_2,u_2)$:
  \begin{alignat*}{2}
	& &&W(\iota_{u_1\oplus u_2}(\lambda, x\oplus y)\,\delta_F \\
    \simeq& &&
	\left(
	  \omega(-2^{-1}y(x)\beta(u_1,u_1)+\lambda \beta(u_1,u_1)+\beta(u_1,\pi_1 Fx) ) \delta_{\pi_1 F+u_1\otimes y}
	\right)\\
	&\otimes&&\left(
	  \omega(-2^{-1}y(x)\beta(u_2,u_2)+\lambda \beta(u_2,u_2)+\beta(u_2,\pi_2 Fx) ) \delta_{\pi_2 F+u_2\otimes y}
	\right) \\
	=&
	 && W(\iota_{u_1}(\lambda,x\oplus y))\otimes W(\iota_{u_2}(\lambda,x\oplus y))\delta_{\pi_1 F}\otimes\delta_{\pi_2 F}.
\end{alignat*}

  The second part is proven analogously to Corollary~\ref{cor:full factorization}.
\end{proof}

The lemma gives a correspondence between the tensor powers of the symplectic group and the tensor powers of the Clifford group.
Assume that $t$ is not a multiple of $p$.
Let $f_i$ be the standard orthonormal basis of $\FF_q$.
Then $u=\sum_{i=1}^t f_i$ is not isotropic, so we can decompose
\begin{align*}
  \FF_q^t = \langle u \rangle \oplus u^\perp=:U_1 \oplus U_2,
  \qquad
  d(U_1) = d(U_2) = t.
\end{align*}
Then Lemma~\ref{lem:jacobi factorization} gives
\begin{align}\label{eqn:symplectoclifford}
  (\cliff_V)^{\otimes t}
  \simeq
  \cliff^{(u)}_{\FF_q^t\otimes V}
  =
  \cliff_V^{(t)} \otimes \mu_{U_2\otimes V}
  \simeq
  \cliff_V^{(t)}
  \otimes
  \mu_V^{\otimes (t-2)}
  \otimes
  \mu_V^{(t)},
\end{align}
where we have used that $u_2=0$ and that $W_{U_2\otimes V}(\iota_0(\lambda, v))=W_{U_2\otimes V}(0)=\Id$.
In Eq.~\eqref{eqn:symplectoclifford}, the action of the Heisenberg group has been compressed to the first tensor factor.
This yields:

\begin{proposition}\label{prop:symplectoclifford}
  Let $u \in U$ be non-isotropic, let $U'=u^\perp$.
  There is a one-one correspondence between
  \begin{enumerate}
	\item
	  representation spaces of the symplectic group acting via $\mu_{U' \otimes V}$ on $L^2(\Hom(X \to U'))$, and
	\item
	  representation spaces of the Jacobi group acting via $\cliff^{(u)}_{U\otimes V}$ on $L^2(\Hom(X\to U))$.
  \end{enumerate}

  In particular, if $p$ does not divide $t$, there is a one-one correspondence between irreducible subrepresentations of
  $(\mu_V)^{\otimes (t-2)}\otimes \mu_V^{(t)}$
  and irreducible subrepresentations of
  $\cliff^{\otimes t}$.
\end{proposition}

\begin{proof}
  As $u$ is non-isotropic, we have the orthogonal direct sum
  \begin{align*}
	U = \left( \FF_q\,u \right) \oplus U'.
  \end{align*}
  As in the proof of Corollary~\ref{cor:full factorization},
  the ismorphism
  \begin{align*}
	i: L^2(\Hom(X\to U)) \to L^2(X^*)\otimes L^2(\Hom(X \to U')
  \end{align*}
  defined by
  \begin{align*}
	\delta_F \mapsto \delta_{u^T F}\otimes\delta_{\pi_2 F},
  \end{align*}
  realizes
  \begin{align}\label{eqn:symplectoclifford general}
	\cliff_{U\otimes V}
	\simeq
	\cliff_V^{(t)} \otimes \mu_{U'\otimes V}.
  \end{align}

  In the one direction, let $\mathcal{K} \subset L^2(\Hom(X \to U'))$ be invariant under $\mu_{U'\otimes V}$.
  Then
  \begin{align}\label{eqn:symplectoclifford correspondence}
	\mathcal{K}':= L^2(X^*)\otimes \mathcal{K} 
  \end{align}
  is invariant under
  $\cliff_V^{(t)} \otimes \mu_{U'\otimes V}$.
  In the other direction, let
  \begin{align*}
	\mathcal{K}' \subset L^2(X^*)\otimes L^2(\Hom(X \to U'))
  \end{align*}
  be invariant under
  $\cliff_V^{(t)} \otimes \mu_{U'\otimes V}$.
  Then, because the Weyl representation acting on the first tensor factor is irreducible, $\mathcal{K}'$ must factorize as
  \begin{align*}
	\mathcal{K}'=L^2(X^*)\otimes \mathcal{K}
  \end{align*}
  with a suitable $\mathcal{K} \subset L^2(\Hom(X \to U'))$ invariant under $\mu_{U'\otimes V}$.
  Thus Eq.~\eqref{eqn:symplectoclifford correspondence} defines a one-one correspondence $\mathcal{K}\to\mathcal{K}'$ as advertised.

  For the second part, assume that $p$ does not divide $t$.
  Let $U=\FF_q^t$ with standard basis $\{f_i\}_i$, and set $u = \sum_{i=1}^t f_i$.
  Then $\beta(u,u)=t$, which is non-zero by assumption.
  The claim now follows from the first part and Lemma~\ref{lem:jacobi factorization}.
\end{proof}

If $t$ is a multiple of $p$, the situation is more complicated. In that case, $u=\sum_if_i$
is isotropic, which reflects the fact that in this case the representation 
\begin{align*}
  (\lambda,x\oplus y)\mapsto W_V^{\otimes t}(\lambda,x\oplus y)
\end{align*}
is Abelian. The smallest non-degenerate subspace $U_1\subset \FF_q^t$ containing 
$u=\sum_i f_i$ is then a hyperbolic plane.
Following the same recipe as above, we can therefore arrange for the Heisenberg group to 
act only on the first \emph{two} copies of $L^2(X^*)$. 
However, the action of the Clifford group on these two copies is its adjoint action (as in
Lemma~\ref{lem:hyperbolic plane}). This action,
unlike the case treated in Proposition~\ref{prop:symplectoclifford}, is reducible.
We will analyze this situation elsewhere \cite{in-prep}.

We close this section with a sample application of Proposition~\ref{prop:symplectoclifford}.
Our goal is to directly see the equivalence of two well-known facts:
(1) The Clifford group forms a unitary $2$-design \cite{dankert_exact_2009,  gross_evenly_2007}, i.e.\ its second tensor power decomposes into a direct sum of two irreducible representations, supported on the $q^n (q^n+1)/2$-dimensional symmetric subspace, and the $q^n(q^n-1)/2$-dimensional anti-symmetric one.
Here, the (anti-)symmetry is w.r.t.\ to an exchange of tensor factors.
(2) The Weil representation of the symplectic group decomposes as the direct sum of two irreducible spaces, namely the $(q^n+1)/2$-dimensional subspace of $L^2(X^*)$ of functions that are symmetric under the reflection $y \mapsto - y$, and the $(q^n-1)/2$-dimensional subspace of anti-symmetric functions.
The correspondence in Proposition~\ref{prop:symplectoclifford} maps these two decompositions onto each other:
\begin{align*}
    \{
	  \text{(anti-)symm. tensors on }&L^2(X^*)\otimes  L^2(X^*)
	\}\\
	\updownarrow&\\
	L^2(X^*)\otimes
    \{
  	  \text{(anti-)symm. }&\text{functions on }L^2(X^*)
	\}.
\end{align*}
As a consistency check: the ortho-complement of $u = f_1 + f_2$ is spanned by $v=f_1 - f_2$.
The interchange of tensor factors acts trivially on $u$, but changes the sign of $v$.
Thus, the two notions of (anti-)symmetry are indeed mapped onto each other.

\section{Summary and Outlook}

Reference~\cite{gh17} introduced a notion of rank for $\Sp(V)$-representations, and showed that there is a one-one correspondence between irreps of $O(U)$ and highest-rank $\Sp(V)$-irreps in $\mu_{U\otimes V}$.
Here, we have classified the rank-deficient components and have achieved a decomposition of $\mu_{U\otimes V}$ in terms of irreducible and inequivalent $\Sp(V)$-representations.

A number of natural directions deserve further attention.
Most importantly from the point of view of quantum information theory, one must treat the case of characteristic $2$.
We will pursue this in an upcoming paper, which will also be written in a language better-suited for consumption by physicists \cite{in-prep}.
While we have occassionally remarked on connections between this paper and previous works from quantum information (e.g.\ Ref.~\cite{grossnezamiwalter}) and coding theory (e.g.\ Ref.~\cite{nebe-book}), the relation between their respective approaches and the one taken in this paper should be made more explicit.
Lastly, the joint action of $O(U)\times \Sp(V)$ should  be worked out more explicitly.

\appendix

\section{Deferred Proofs}

\subsection{Factorization property of the oscillator representation}
\label{app:factorization proof}

It is possible to prove Lemma~\ref{lem:factorization general} by directly verifying the claim on a set of generators (as in Eqs.~\eqref{eqn:N}, \eqref{eqn:J}, and \eqref{eqn:D}) for $\Sp(V_1)\times \Sp(V_2)$.
Our approach is based on realizing that it suffices to check the factorization property for Weyl operators (as in Eq.~\eqref{eqn:weyl operators}), and then use Eq.~\eqref{eqn:symplectic action} to ``lift'' it to the oscillator representation.

\begin{proof}[Proof of Lemma~\ref{lem:factorization general}]
  Assume that
  $X=X_1\oplus X_2$
  and that
  $V=V_1\oplus V_2$
  is the resulting decomposition of $V$.
  By computing the action on basis vectors, it is immediate that
  \begin{align*}
	  i \, W_V(v) \, i^{-1} =  W_{V_1}(v_1)\otimes W_{V_2}(v_2),
  \end{align*}
  where
  \begin{align*}
	i: L^2(X^*) \to L^2(X^*_1) \otimes L^2(X^*_2),
	\qquad
	\delta_y \mapsto \delta_{y \pi_1}\otimes\delta_{y\pi_2}
  \end{align*}
  is the isomoprhism introduced in Eq.~\eqref{eqn:isomorphism factorization}.

  Let $S\in\Sp(V_1)$, then, using Eq.~\eqref{eqn:symplectic action},
  \begin{align*}
	&\big(i \, \mu_V(S) \, i^{-1}\big)
	\,
	\big(i \, W(v) \, i^{-1}\big)
	\,
	\big(i \, \mu_V(S)^\dagger\,i^{-1}\big) \\
  =&
	i \, W(Sv)\, i^{-1}\\
		=& 
		W_{V_1}(Sv_1)\otimes W_{V_2}(v_2)\\
		=& 
		\Big(\mu_{V_1}(S_1)\otimes \ii\Big) \Big( W_{V_1}(v_1)\otimes W_{V_2}(v_2)\Big) \Big(\mu_{V_1}(S_1)\otimes \ii\Bigl)^\dagger.
  \end{align*}
  Since the Weyl operators $W_V(v)$ for a basis for $\End(L^2(X^*))$, this implies
  \begin{align*}
	  i \,\mu_V(S_1)\, i^{-1} = \kappa(S) (\mu_{V_1}(S)\otimes \ii)
  \end{align*}
  for some scalar function $\kappa:\Sp(V_1)\to\CC^\times$. We now show $\kappa(S)=1$ for all $S$.

  Unitarity implies that $|\kappa(S)|=1$. Because
  \begin{align*}
	  i\,\mu_V|_{\Sp(V_1)}\, i^{-1}
  \end{align*}
  is a representation,
  $\kappa$ must also be a (one dimensional) representation of $\Sp(V_1)$. Let $\mcal{N}_1\rtimes\mcal{D}_1$ be
  the Siegel parabolic of $\Sp(V_1)$ with $\mcal{N}_1$ its unipotent radical.

  Since $\kappa$ is a one dimensional representation, it must contain only one weight associated
  to $\mcal{N}_1$. This $\mathcal{N}_1$-weight must have a trivial orbit under conjugation by $\mcal{D}_1\cong\Gl(X_1)$
  transformations, and thus $\kappa|_{\mcal{N}_1}=1$. But $\Sp(V_1)$ is generated by $\mcal{N}_1$-conjugates,
  so $\kappa = 1$.
\end{proof}

\subsection{Fourier transforms and invariance}
\label{app:ft invariance}

Here, we prove Lemma~\ref{lem:ft invariance lemma}.

We begin by noting that $B$ turns $\Hom(X\to U)$ into an orthogonal space with form $\beta_B$ given by
\begin{align*}
	\beta_B(F,G) = \tr F^t \beta G B.
\end{align*}
Let $W=\Hom(X\to U')$.
In the following we will show that
\begin{enumerate}
\item for any $\Phi\in L^2\Hom(X\to U)$, $\supp\Phi\subseteq W$ if and only if $\mu_{U\otimes V}(J_B)\Phi$ is invariant under $W^\perp$ translations,
\item $W^\perp = \Hom(X\to U'^\perp)$.
\end{enumerate}
These two claims imply the first statement of the lemma

For the first claim, start with the ``only if'' direction.
Apply the inverse map (associated with $-B$) to a function $\tilde\Phi$ with the invariance stated.
Then
\begin{align*}
\Phi(F)
&=
\gamma(B,U)
\sum_{F'} \omega( -\beta_B(F, F'))\tilde\Phi(F') \\
&=
\gamma(B,U)
\sum_{C \in \Hom(X\to U) / W^\perp}
  \tilde\Phi(C)
\sum_{G \in W^\perp }
\omega( -\beta_B(F, C+G)) \\
&=
\gamma(B,U)
\sum_{C \in \Hom(X\to U) / W^\perp}
  \tilde\Phi(C)
\omega( -\beta_B(F, C))
\sum_{G \in W^\perp }
\omega( -\beta_B(F, G)) \\
&=
\gamma(B,U)
|W^\perp| \delta_{W}(F)\,
\sum_{C \in \Hom(X\to U) / W^\perp}
  \tilde\Phi(C)
\omega( -\beta_B(F, C)).
\end{align*}

Conversely, the set of $\Phi$'s with support in $W$ is a vector space of dimension $|W|$.
At the same time, the set of solutions we have identified in the direct direction has dimension
\begin{align*}
\left| \Hom(X\to U) / W^\perp \right|
=
\frac{| \Hom(X\to U)|}{|W^\perp|}
=
\frac{| \Hom(X\to U)|\,|W|}{|\Hom(X\to U)|}
=|W|,
\end{align*}
so we have found all solutions.

Now we prove the second claim. Assume $F$ is such that $\beta_{B}(F,F')=0$ for all $F'\in\Hom(X \to U')$,
and choose $F'=u\otimes y$ for $y\in X^*$ and $u\in U'$.
Note that
\begin{align*}
	F^t\beta F' B = F^t\big(\beta(u)\big)\otimes (By),
\end{align*}
where we used the fact that $B$ is symmetric. Hence
\begin{align*}
	\beta_B(F,F')=\big(F^t\beta(u)\big)(By) = \beta(u, FB y).
\end{align*}
Because $y\in X^*$ and $u\in U'$ are arbitrary, and because $B$ is surjective, it follows that $F\in \Hom(X\to U'^\perp)$
and with this the claim also follows.

Now on to the second statement of the lemma.
Acting explicitly on the indicator function of $W$ we get
\begin{align*}
	\mu_{U\otimes V}(J_B)\sum_{F\in W}\delta_F	&= \gamma(B)^{-1}\sum_{F\in W}\sum_{F'\in W^\perp} \omega^{\beta_B(F,F')}\delta_{F'},
\end{align*}
where we used the first statement to restrict the sum over $F'$. Notice that by the definition of $W^\perp$, every coefficient in the
expression above is 1, so that
\begin{align*}
	\mu_{U\otimes V}(J_B)\sum_{F\in W}\delta_F = \gamma(B)^{-1}|W| \sum_{F'\in W^\perp}\delta_{F'},
\end{align*}
as claimed.

\subsection{Contiguity of ranks}
\label{sec:rank spectrum proof}

In this section we prove Proposition~\ref{prop:rank spectrum}.
Our strategy will be to find a set of generators for $\Sp(V)$ which consists of elements that either keep the rank of an $\mathcal{N}$-weight invariant, or change it by at most one.
It then follows that if the ranks of the $\mathcal{N}$-spectrum has a gap, then the representation is reducible.

Recall that $\rep$ is an irreducible representation of rank $<t$, and
\begin{align*}
	R:=\{ \,k\,|\,\text{there is an } \mcal{N}\text{-weight with rank } k\}.
\end{align*}

Let $\rep^r$ be the subspace of $\mathcal{K}\subset L^2(\Hom(X\to U))$ that is spanned by $\mcal{N}$-weights of  rank $r$.
This space is invariant under the action of the parabolic subgroup $\mcal{N}\rtimes\mcal{J}$.
We will analyze its image under the Fourier transforms $\mcal{J}$.

Let $\{e_i\}$ be an arbitrary basis of $X$ and $\eps_i$ be its dual.
For any isomorphism $B: X^*\to X$,
there is a $C\in \Gl(X)$ satisfying
\begin{align*}
	C B \eps_i = e_i, \quad \forall\, i.
\end{align*}
Using the isomorphism in Eq.~\eqref{eqn:isomorphism factorization},
we find that there exist a set of $B_i: \langle \eps_i\rangle \to \langle e_i\rangle$ for which
\begin{align*}
	i\, \mu_{U\otimes V}(C)\mu_{U\otimes V}(J_B)\, i^{-1} 
  = 
  \mu_{U\otimes V_1}(J_{B_1})\otimes \cdots \otimes \mu_{U\otimes V_n}(J_{B_n}),
\end{align*}
where $V_i := \langle\eps_i,e_i\rangle$. It follows that $\Sp(V)$ is generated by the parabolic subgroup
together with any
$i^{-1}\, (\mu_{U\otimes V_1}(J_{B_1})\otimes \ii) \, i$
(sometimes referred to as \emph{single-system Fourier transform} in quantum information theory).
Let $X_1=\langle e_1\rangle$, $X_2=\langle e_2,\,...\,,\,e_n\rangle$, and let $V=V_1\oplus V_2$ be the corresponding decomposition of phase space.
Let $\pi_1$, $\pi_2$ be the projections associated with the decomposition $\Hom(X\to U)=\Hom(X_1\to U)\oplus\Hom(X_2\to U)$.
We see that
\begin{align*}
	\big(i^{-1}\, (\mu_{U\otimes V_1}(J_{B_1})\otimes \ii) \, i\big)\delta_F 
  = 
  \gamma^{-1}(B_1) \sum_{F' \in \Hom(X_1\to U)} \omega^{\beta_B(F',\pi_1 F)}\delta_{\pi_2 F + F'}.
\end{align*}
Throughout the rest of the argument, let $\rank F^TF =r$.

Now, $\range\,\pi_2 F$ is either equal to $\range F$ or it is a subspace of the latter with co-dimension 1. Thus
$\rank (\pi_2 F)^T(\pi_2 F) \in\{r, r -1\}$. Furthermore, either $\range\,\pi_2 F + F' = \range\,\pi_2 F$ or
\begin{align*}
	\range\,\pi_2 F \subset \range\,\pi_2 F + F'
\end{align*}
is a subspace of co-dimension 1. Thus,
\begin{align*}
	\rank (\pi_2 F + F')^T(\pi_2 F + F') \in \{r-1,\, r,\, r+1\},
\end{align*}
for any $F'\in\Hom(X_1\to U)$. This implies that
\begin{align*}
	i^{-1}\, (\mu_{U\otimes V_1}(J_{B_1})\otimes \ii) \, i : \rep^r \to \rep^{r-1}+\rep^r+\rep^{r+1}.
\end{align*}

If for some $r\in R$ it
held that $\mcal{K}^l=\{0\}$ then the spaces $\sum_{r>l}\rep^r$ and $\sum_{r<l}\rep^r$ would be invariant under all
generators (and thus subrepresentations). Since $\rep$ is irreducible by assumptiuon, this cannot happen.

\bibliography{MontealegreMora-Gross-rank}

\end{document}